\documentclass[reqno, letter]{amsart}
\pdfoutput=1
\usepackage{amsmath, amssymb,amsthm, verbatim, comment, color
}

\usepackage{changes}
\usepackage{lipsum}

\usepackage{txfonts}
\usepackage{enumerate}
\usepackage{color}

\usepackage{ hyperref}
\usepackage{caption}
\usepackage{accents}

\newtheorem{theorem}{Theorem}

\newtheorem{corollary}[theorem]{Corollary}
\newtheorem{definition}[theorem]{Definition}
\newtheorem{lemma}[theorem]{Lemma}
\newtheorem{proposition}[theorem]{Proposition}

\newtheorem{claim}[theorem]{Claim}
\theoremstyle{remark}
\newtheorem{remark}[theorem]{Remark}
\newtheorem{example}[theorem]{Example}

 \renewcommand{\phi}{\varphi}

\newcommand{\E}{\mathbb{E}}
\renewcommand{\P}{\mathbb{P}}
\newcommand{\N}{\mathbb{N}}

\newcommand{\R}{\mathbb{R}}
\newcommand{\C}{\mathbb{C}}

\newcommand{\be}{\begin{equation}}
\newcommand{\ee}{\end{equation}}

\DeclareMathOperator{\supp}{supp}
\newcommand{\bea}{\begin{eqnarray}}
\newcommand{\bes}{\begin{subequations}}
\newcommand{\ees}{\end{subequations}}
\newcommand{\bgt}{\begin{gather}}
\newcommand{\egt}{\begin{gather}}
\newcommand{\eea}{\end{eqnarray}}
\newcommand{\beaa}{\begin{eqnarray*}}
\newcommand{\eeaa}{\end{eqnarray*}}

\usepackage{bbm}

\renewcommand{\epsilon}{\varepsilon}

\newcommand{\fourIdx}[5]{%
\setbox1=\hbox{\ensuremath{^{#1}}}%
 \setbox2=\hbox{\ensuremath{_{#2}}}%
 \setbox5=\hbox{\ensuremath{#5}}%
 \hspace{\ifnum\wd1>\wd2\wd1\else\wd2\fi}%
 \ensuremath{\copy5^{\hspace{-\wd1}\hspace{-\wd5}#1\hspace{\wd5}#3}%
 _{\hspace{-\wd2}\hspace{-\wd5}#2\hspace{\wd5}#4}%
 }}

\numberwithin{equation}{section}
\numberwithin{theorem}{section}

\renewcommand{\subset}{\subseteq}
\renewcommand{\supset}{\supseteq}

\newcommand\dist{\hbox{\rm dist}}

\usepackage{subcaption}
\usepackage{asymptote}
\usepackage{graphicx}
\begin{asydef}
usepackage("amsmath");
usepackage("txfonts");
texpreamble("");
import fontsize;
defaultpen(fontsize(10pt));
\end{asydef}

\begin{document}

\title{Structure of optimal martingale transport plans in general dimensions}


\author{Nassif Ghoussoub}
\address{Nassif Ghoussoub: Department of Mathematics\\ University of British Columbia\\ Vancouver, V6T 1Z2 Canada}
\email{nassif@math.ubc.ca}
\author{Young-Heon Kim}
\address{Young-Heon Kim: Department of Mathematics\\ University of British Columbia\\ Vancouver, V6T 1Z2 Canada\\
\newline \& School of Mathematics, Korea Institute for Advanced Study, Seoul, Korea.}
\email{yhkim@math.ubc.ca}
\author{Tongseok Lim}
\address{Tongseok Lim: Department of Mathematics\\ University of British Columbia\\ Vancouver, V6T 1Z2 Canada}
\email{lds@math.ubc.ca}
\date{Revised March 15, 2016.}

\thanks{The two first-named authors are partially supported by grants from the 
Natural Sciences and Engineering Research Council of Canada (NSERC). Y. H. Kim is also supported 
by an Alfred P. Sloan Research Fellowship.
T. Lim is partly supported by a doctoral graduate fellowship from the University of British Columbia. 
Part of this research has been done while the authors  were visiting 
the Fields institute in Toronto for the thematic program on ``Calculus of Variations'' in Fall 2014. We are thankful for the hospitality and the great research environment that the institute provided. 
\copyright 2015 by the authors.
}

\begin{abstract} Given two probability measures $\mu$ and $\nu$ in ``convex order" on $\R^d$, we study the profile of  one-step martingale plans $\pi$ on $\R^d\times \R^d$ that optimize the expected value of the modulus of their increment among all martingales having $\mu$ and $\nu$ as marginals. While there is a great deal of results for the real line (i.e., when $d=1$),  much less is known in the richer and more delicate higher dimensional case that we tackle in this paper.  We show that many structural results can be obtained whenever a natural dual optimization problem is attained, provided the initial measure $\mu$ is absolutely continuous with respect to the Lebesgue measure. One such a property is that $\mu$-almost every $x$ in $\R^d$ is transported by the optimal martingale plan into a probability measure $\pi_x$ concentrated on the  extreme points of the closed convex hull of its support. This will be established in full generality in the 2-dimensional case, and also for any $d\geq 3$ as long as the marginals are in ``subharmonic order". In some cases, $\pi_x$ is supported on the vertices of a $k(x)$-dimensional polytope, such as when the target measure is discrete.
Many of the proofs rely on a remarkable decomposition of ``martingale supporting'' Borel subsets of $\R^d\times \R^d$ into a collection of mutually disjoint components by means of a ``convex paving" of the source space.  If the martingale is optimal, then each of the components in the decomposition supports a restricted optimal martingale transport for which the dual problem is attained. These decompositions are used to obtain structural results in cases where duality is not attained.  On the other hand, they can also be related to higher dimensional Nikodym sets.  

\end{abstract}
\maketitle

\tableofcontents

\section{Introduction}

We study the profile  
 of one-step martingales $\pi$ on $\R^d\times \R^d$ that optimize the expected value of the modulus of their increment, among all martingales with two given marginals $\mu$ and $\nu$ in convex order. More precisely, we  investigate the structure of conditional probabilities
 $(\pi_x)_{x \in \supp \mu}$ 
 on $\R^d$ which describe how a given particle at $x$ is propagated under such transport plans. 
These questions originate in mathematical finance and are variations on the original  Monge-Kantorovich 
problem, where one considers all couplings of the given marginals and not only those of martingale type \cite{{Mo1781}}, \cite{gm}, \cite{Vi03, Vi09}. However, unlike solutions of the Monge-Kantorovich
problem, which are often supported on graphs (such as the well-known Brenier solution \cite{br} for the cost given by the squared distance), the additional martingale constraint forces the transport to split the elements of the initial measure $\mu$. One cannot therefore expect --but in trivial cases-- that optimal martingale plans be supported on graphs. 

These questions are motivated by problems in mathematical  finance, which call for no-arbitrage lower  (or upper) bounds on the price of a forward starting straddle, given today's vanilla call prices at the two relevant maturities. Just like in the Monge-Kantorovich theory for optimal transport, these problems have dual counterparts, whose financial interpretation amount to constructing the most (or least) expensive semi-static hedging strategy which sub-replicates the payoff of the forward starting straddle for any realization of the underlying forward price process.

The minimization and maximization problems are quite different, though by now well understood  when the marginals are probability measures on the real line, at least in the case of one-step martingales. 
 We refer to Hobson-Neuberger \cite{HoNe11}, Hobson-Klimmek \cite{HoKl12}, and  Beiglb{\"o}ck-Juillet \cite{bj}. For the multi-step case, see Beiglb{\"o}ck-Henry-Labordere-Penkner \cite{BeHePe11}. The dynamic case have been also studied by  Galichon-Henri-Labordere-Touzi \cite{GaHeTo11} and Dolinsky--Soner \cite{ds1, ds2}.   The two cases studied are when the cost is either $c(x, y)=|x-y|$, which is the main focus of this paper, or the case when the cost satisfies the so-called Mirlees condition. Note that the one-dimensional case is closely related to Skorohod embedding problems \cite{Obloj}, since real valued martingales can be realized as adequately stopped Brownian paths. See for example Hobson \cite{Ho11}, Beiglb{\"o}ck-Cox-Huesmann \cite{bch} and Beiglb{\"o}ck-Henry-Labordere-Touzi \cite{BeHeTo}.

Surprisingly, much less is known in the case where the marginals are supported on higher dimensional Euclidean spaces $\R^d$. In this direction,  Lim \cite{Lim} considered the optimal martingale transport problem under radially symmetric marginals on $\R^d$, while Ghoussoub-Kim-Lim consider in \cite{GKL1} the corresponding optimal Skorokhod embedding. 
 In this paper, we shall tackle the following general optimization problem associated to a cost function $c:\R^d\times \R^d \to \R$:
\begin{equation}\label{MGTP}
\mbox{Maximize / Minimize} \,\,\, \text{cost}[\pi] = \int_{\R^d\times \R^d} c(x,y) \,d\pi(x,y)\quad\mbox{over}\quad \pi\in MT(\mu,\nu).
\end{equation}
Here $MT(\mu,\nu)$ is the set of {\it martingale transport plans}, that is the set of probabilities $\pi$
on $\R^d \times \R^d$ with marginals $\mu$ and $\nu$, such that for $\mu$-almost $x\in\R^d$, 
the component $\pi_x$ of its disintegration $(\pi_x)_x$ with respect to $\mu$, i.e. $d\pi(x,y)=d\pi_x(y)d\mu(x)$, has its   barycenter at $x$; in other words, for any convex function  $\varphi$ on $\R^d$, one has
$\varphi(x) \le \int_{\R^d} \varphi(y)\,d\pi_x (y).$\\
One can also use the probabilistic notation, which amounts to 
\begin{align}\label{opt}
\text{Maximize / Minimize}\quad  
\E_{\rm P} \,c(X,Y)
\end{align}
over all martingales $(X,Y)$ on a probability space $(\Omega, {\mathcal F}, P)$ into $\R^d \times \R^d$ (i.e. $E[Y|X]=X$) with laws $X \sim \mu$ and $Y \sim \nu$ (i.e., $P(X\in A)=\mu (A)$ and $P(Y\in A)=\nu (A)$ for all Borel set $A$ in $\R^d$). Note that in this case, the disintegration of $\pi$ can be written as the conditional probability   $\pi_x (A) = \P(Y\in A|X=x)$. 

A classical theorem of Strassen \cite{St65} states that the set $MT(\mu, \nu)$ of martingale transports is nonempty if and only if the marginals $\mu$ and $\nu$ are in {\it convex order}, that is if
\begin{enumerate}
\item $\mu$ and $\nu$ are probability measures with finite first moments, and 
\item  $\int_{\R^d} \varphi\,d\mu\leq \int_{\R^d} \varphi\, d\nu$ for every convex function $\varphi$  on $\R^d$.
\end{enumerate}
In that case we will write $\mu\leq_C\nu$, which is sometimes called the {\it Choquet order for convex functions}. Note that $x$ is the barycenter of a measure $\nu$ if and only if $\delta_x \leq_C\nu$, where $\delta_x$ is Dirac measure at $x$. .

We will mostly consider the Euclidean distance cost
$c(x,y) = |x - y|$
unless stated otherwise, although 
some of the results below hold for more general costs. We shall use the term optimization in problem (\ref{MGTP}) whenever the result holds for either maximization  
or minimization. 
We shall  be more specific otherwise, since  it will soon become very clear that the two cases can sometimes be fundamentally different.
The following theorem summarizes the main structural result when $\mu$ and $\nu$ are one-dimensional marginals. Hobson-Neuberger \cite{HoNe11} were first to deal with the maximization case while Beiglb\"ock-Juillet \cite{bj} and  D.~Hobson and M.~Klimmek \cite{HoKl12} deal with the context of minimization.

\begin{theorem} \label{beju} {\rm (Beiglb\"ock-Juillet \cite{bj}, Hobson-Neuberger \cite{HoNe11}, Hobson-Klimmek \cite{HoKl12})} Assume that $\mu$ and $\nu$ are probability measures in convex order on $\R$, and that $\mu$ is continuous.
There exists then a unique optimal martingale transport plan  $\pi$ for the cost function 
$c(x, y) =|x-y|$, 
such that:
 
 \begin{enumerate}
  \item If $\pi$ is a minimizer,  then its disintegration satisfies $|\supp \pi_x| \leq 3$ for every $x \in \R$. More precisely, $\pi$ can be decomposed into $\pi_{stay} +\pi_{go}$, where $\pi_{stay} = (Id \circ\times Id)_\#(\mu \wedge \nu)$ (this measure is concentrated on the diagonal of $\R^2$) and $\pi_{go}$ is concentrated on ${\rm graph}(T_1)\cup {\rm graph}(T_2)$
where $T_1, T_2$ are two real-valued functions.
\item   If $\pi$ is a maximizer,  then its disintegration satisfies $|\supp \pi_x| \leq 2$ for every $x \in \R$, and $\pi$ is concentrated on ${\rm graph}(T_1)\cup {\rm graph}(T_2)$
where $T_1, T_2$ are two real-valued functions.
\end{enumerate}

\end{theorem} 

Our main goal in this paper is to consider higher dimensional analogues of the above result. 
In \cite{Lim},  Lim showed that the above theorem extends, in the case of minimization, to the setting where the marginals are radially symmetric on $\R^d$ and $c(x, y)=|x-y|^p$ for $0<p\leq 1$.
The general case is wide open and our goal is to work towards establishing the following:\\

\noindent {\bf Conjecture 1:}  {\it Consider the cost function 
$c(x, y) =|x-y|$  and assume that $\mu$ is absolutely continuous with respect to the Lebesgue measure on $\R^d$ ($\mu << \mathcal{L}^d$). If $\pi$ is a martingale transport that optimizes $(\ref{MGTP})$.  Then for $\mu$-almost every $x$, $\supp \pi_x$ coincides with the  set of extreme points of the convex hull of $\supp \pi_x$, i.e., 
${\rm supp}\, \pi_x = {\rm Ext}\,\bigg{(} {\rm conv}({\rm supp}\, \pi_x) \bigg{)}$.
}
\begin{remark} If $\supp \pi_x$ is bounded for $\mu$-almost all $x$ (which is the case in particular when the target measure $\nu$ is compactly supported), then ${\rm conv}({\rm supp}\, \pi_x) = \overline{{\rm conv}}({\rm supp}\, \pi_x)$. In this case, the set of extreme points $ {\rm Ext}\,\bigg{(} \overline{{\rm conv}}({\rm supp}\, \pi_x) \bigg{)}$ is also called the Choquet boundary of the compact convex set $\overline{{\rm conv}}({\rm supp}\, \pi_x)$. Our conjecture can therefore be rephrased as: For $\mu$ a.e. $x$, $\supp \pi_x$ is equal to the Choquet boundary of its closed convex hull.
\end{remark}

Note that for the minimization problem,  
we can and will assume that $\mu \wedge \nu =0$ 
since any minimizing martingale transport for problem \eqref{MGTP} must let the support of $\mu \wedge \nu$ stay put. See \cite{bj} or \cite{Lim} for a proof.
  One can then easily see that in the one dimensional case, the above conjecture reduces to Theorem \ref{beju} since then the dimension of the linear span of $\supp \pi_x$ is one and the Choquet boundary consists of exactly two points, unless of course $\supp \pi_x$ is a singleton. 

We shall be able to prove the above conjecture in many important cases, in particular, when a natural dual optimization problem is attained (Theorem \ref{thm: main extremal}),   
or when the linear span of  $\supp \pi_x$ has full dimension (Corollary \ref{cor: dim d-1}). Another case where the answer is affirmative is in dimension $d=2$ (Theorem \ref{th: Choquet 2d.0}) provided the second marginal has compact support. The conjecture also holds partially (Theorem \ref{cor: local Choquet in SH}) when the marginals are in ``subharmonic order," that is if 
$$\int_{\R^d} \varphi\,d\mu\leq \int_{\R^d} \varphi\, d\nu \quad \hbox{for every subharmonic function $\varphi$  on $\R^d$.}
$$

We actually expect to have a more rigid structure in the case of minimization.  Indeed, Lim \cite{Lim} showed that  in this case, assuming $\mu\wedge \nu =0$, we also have $| \supp \pi_x | \le 2$ for $\mu$-almost all $x$,  
 whenever the marginals are radially symmetric on $\R^d$ and $c(x, y)=|x-y|^p$ for $0<p\leq 1$.  The general case remains open as we propose the following:\\
 
 \noindent {\bf Conjecture 2 (Minimization):}  {\it  Consider the cost function 
$c(x, y) =|x-y|$  and assume that $\mu$ is absolutely continuous with respect to the Lebesgue measure on $\R^d$, that  $\mu \wedge \nu =0$. If $\pi$ is a martingale transport that minimizes $(\ref{MGTP})$. Then for $\mu$ almost every x, the set $\supp \pi_x$ consists of  
 $k+1$ points  
that form the vertices of a $k$-dimensional polytope,  where $k:=k(x)$ is the dimension of the linear span of $\supp \pi_x$
and therefore, the minimizing solution is unique.}\\

We shall give a partial answer to the above conjecture under the assumption that the target measure $\nu$ is discrete. Actually, in this case the result holds true in both the maximization and minimization cases  (Theorem \ref{th: finite images}).  
 We note however that --unlike the minimization case-- one cannot always expect in higher dimensions neither the uniqueness of a maximizer (Example \ref{ex: nonunique max}), nor a polytope-type structure for 
  $\supp \pi_x$ (Example \ref{rm: no finite for max}), even when the marginals are radially symmetric.

Just like in the Monge-Kantorovich theory, the above optimization problem (\ref{MGTP}) has a dual formulation, which will be crucial to our analysis. And similarly to that theory, the dual problem can be studied independently of the primal problem and without any underlying reference measures. Recall that for the quadratic cost studied by Brenier, the dual problem amounts to considering convex functions $\beta$, their Fenchel-Legendre duals $\alpha:=\beta^*$ and the set  $\Gamma =\{(x,y)\in \R^d\times \R^d; \beta(y) +\alpha (x)=\langle x, y\rangle\}$, which happens to be the graph of the subdifferential of $\beta$. Similar but more complicated phenomena arise in our situation. We shall work with the following notions.

For a subset $\Gamma$ in $\R^d\times \R^d$, we shall denote by $\Gamma_x$, the fiber $\Gamma_x:=\{y\in \R^d; \, (x,y)\in \Gamma\}$. For a Borel set $\Gamma \subset \R^d \times \R^d$, 
we write $X_\Gamma :=$ proj$_X \Gamma$, $Y_\Gamma :=$ proj$_Y \Gamma$, i.e. $X_\Gamma$ is the projection of $\Gamma$ on the first coordinate space $\R^d$, and $Y_\Gamma$ on the second. 

\begin{definition}\label{def: E m}  \rm
Let $c: \R^d \times \R^d \to \R$ be a cost function and let $X, Y \subset \R^d$ be two Borel sets. 
 \begin{enumerate}
 \item We say that a triplet of measurable functions $(\alpha, \gamma, \beta)$ is {\it an admissible triple on $X\times Y$,} if $\alpha : X\to \R$, $\beta: Y\to \R$, and $\gamma: X\to \R^d$ satisfy the following inequality
  \begin{equation}\label{eq: min dual condition supp}
 \beta(y)-\alpha(x) -\gamma(x)(y-x) \leq c(x,y)\,\, \hbox{for all $(x, y)\in  X \times Y$.}
 \end{equation}
 We shall denote by $E_m(c, X, Y)$ the set of all such {\em admissible dual  triples}. A similar definition holds when the inequality is reversed, and the set of those triplets will be denote by $E_M(c, X, Y)$. Note that $E_M(c, X, Y)=E_m(-c, X, Y)$.
\item  For an admissible triple $(\alpha, \gamma, \beta)$, we will consider the set where equality holds, that is
 \begin{align}\label{eq: Gamma triple}
 \Gamma_{(\alpha, \gamma, \beta)} : = \{ (x, y) \in X \times Y \ | \ \beta(y)  - \alpha(x)  - \gamma(x) \cdot (y-x) = c(x,y)\}.
\end{align}
We shall sometimes allow $\gamma$ to be a set-valued function. In this case, the above inequality/equality will mean that they actually hold for any vector $b$ in $\gamma (x)$.  
\item Any non-empty subset of $\Gamma_{(\alpha, \gamma, \beta)}$ will be called a {\em $c$-contact layer for $(\alpha, \gamma, \beta)$ in $X\times Y$}. When the ambient space is not specified, it means that it is simply $X_\Gamma \times Y_\Gamma$. \\
We shall sometimes say that a set $\Gamma$ is {\em $c$-exposed by the admissible triple $(\alpha, \gamma, \beta)$} if it is contained in $\Gamma_{(\alpha, \gamma, \beta)}$.
\end{enumerate}
\end{definition}
\noindent Denoting $E_m=E_m(c, \R^d, \R^d)$, one can then show (see for example \cite{BeHePe11}) that if the cost $c$ is lower semi-continuous, then for the minimization problem, 
 \begin{eqnarray}\label{No.gap.1}
&&\min\left\{ \int_{\R^d\times \R^d} c(x,y) \,d\pi;\, \pi\in MT(\mu,\nu)\right\} \\
&& \qquad \qquad   =\sup \left\{\int_{\R^d}\beta d\nu -\int_{\R^d}\alpha d\mu;\,\, (\alpha, \gamma, \beta) \in E_m\, \,  \hbox{for some $\gamma \in C_b(\R^d, \R^d)$} \right\}.
\end{eqnarray}
Similarly, if the cost $c$ is upper semi-continuous, then 
 \begin{eqnarray}\label{No.gap.2}
&&\max\left\{ \int_{\R^d\times \R^d} c(x,y) \,d\pi;\, \pi\in MT(\mu,\nu)\right\} \\
&& \qquad \qquad   =\inf \left\{\int_{\R^d}\beta d\nu -\int_{\R^d}\alpha d\mu;\,\, (\alpha, \gamma, \beta) \in E_M\, \,  \hbox{for some $\gamma \in C_b(\R^d, \R^d)$} \right\}.
\end{eqnarray}
 Note that if $\pi$ is an optimal martingale measure and if the corresponding dual problem is attained on a triplet $(\alpha, \gamma, \beta)$, then it is easy to see that there exists a Borel subset $\Gamma \subset \R^d \times \R^d$ with full $\pi$-measure that is a $c$-contact layer for $(\alpha, \gamma, \beta)$,  namely,  
  \begin{equation}\label{eq: dual.OK}
\beta(y)-\alpha(x) -\gamma(x)(y-x)= c(x,y)\,\, \hbox{if and only if $(x, y)\in \Gamma$.}
 \end{equation}
We shall show that such $c$-contact layers have a specific extremal structure.  As a result, any martingale transport $\pi \in MT(\mu, \nu)$ which is concentrated on a $c$-contact layer, when $c(x,y) = \pm |x-y|$, will satisfy the conjecture (1).

Recall that in the Monge-Kantorovich theory for mass transport, the dual problem is normally attained, and the ``corresponding $c$-contact layer" is a set of the form 
$
\Gamma=\{(x,y); \beta (y)-\alpha (x)=c(x,y)\}, 
$
where $\beta$ and $\alpha$ related through $c$-Legendre duality, which let them inherit some of the regularity properties of $c$. We shall follow a similar methodology here by defining  and exploiting in Section \ref{Martingale.transform} a notion of {\it martingale  $c$-Legendre duality} between the function $\beta$ and the pair $(\alpha, \gamma)$. This will allow us to establish the regularity properties needed to analyze the structure of $c$-layer sets.


However, unlike the Monge-Kantorovich setting, attainment of the dual problem does not often hold for optimal martingale transports --at least in the maximization problem-- even in the one-dimensional case, as shown in \cite{BeHePe11}. See Example~\ref{ah} below. We therefore explore whether dual attainment can happen locally, which is sufficient to imply Conjecture (1). We prove in Section~\ref{S: local dual} 
that it is indeed the case under suitable assumptions on the marginals, such as when they are comparable for the order  induced by subharmonic functions; see Theorem~\ref{cor: local Choquet in SH}.  


More importantly, we then proceed to consider the general case by establishing a remarkable decomposition for any Borel set $\Gamma$ supporting a given optimal martingale transport $\pi$ into disjoint components  $\{\Gamma_C\}_{C\in I}$ in such a way that each piece is a $c$-contact layer for  
an admissible triplet $(\alpha_C, \gamma_C, \beta_C)$. 
What is remarkable is that this decomposition into 
$c$-contact layers can be established  in full generality  (i.e., for any cost function)
 and without any reference to a martingale transport problem or even to any reference measure. The decomposition is done through an equivalence relation on the projection $X_\Gamma$ of $\Gamma$ on the first coordinate,  that is
 induced by a well chosen {\it irreducible convex paving}, that is 
 a collection of mutually disjoint convex subsets in $\R^d$ that covers $X_\Gamma$. 
See  Theorem~\ref{th:k} for the precise statement.

We note that this result can be seen as a generalization of the decomposition of Beiglb\"ock-Juillet  \cite{bj} in the one-dimensional case $d=1$, where the disintegration comes from restricting the measures $\mu, \nu$ onto open subintervals of $\R$ obtained by examining the potential functions for $\mu, \nu$.  Like theirs, our decomposition applies to any cost function $c$ and not only to $c(x, y) =|x-y|$.  
 It is however quite different since it depends on the support of the martingale measure $\pi$ that we start with. More importantly, our decomposition needs not be countable (Example \ref{nikodym}) which creates additional and interesting complications for the higher dimensional cases.   

We shall use the above decomposition to establish the above stated conjectures under various conditions. For example, Conjecture 1 holds in dimension $d=2$ (Theorem~\ref{th: Choquet 2d.0}), and also in the case where the dimensions of all components $(C)_{C\in I}$ are  $d$-dimensional (see Corollary~\ref{cor: dim d-1}).

Remarkably, the results discussed so far do not distinguish between the minimization and maximization problems (except that we assume that $\mu \wedge \nu =0$ in the case of minimization). The  previously mentioned decomposition can be used to prove Conjecture 2) in either the minimization and maximization case, provided  the target measure $\nu$ has a countable support (Theorem~\ref{th: finite images}). However, as mentioned above, we believe  that these two problems are quite different, at least in terms of finding finer structural results for each of the cases. 

Back to the martingale problem, we then consider the disintegration $\{\pi_C\}_{C\in I}$ of any martingale measure $\pi$ along the above described  decomposition of its 
support
 $\Gamma$ (Theorem~\ref{th:disintegration}).
This then gives a canonical decomposition of the optimal martingale problem into a collection of non-interactive martingale problems where  duality is attained for each piece $\pi_C$ in $MT(\mu_C, \nu_C)$. 

In the next section, we give the precise statements of our results. In Section \ref{Martingale.transform}, we introduce and study the notion of {\it martingale  $c$-transforms}, which will be used to improve the regularity properties of admissible triples. This will be 
used in Section \ref{triplet} to analyze the structure of $c$-contact layers that are exposed by such triples. We apply these results in Section \ref{ss: structure from dual}  to the case where the dual problem is attained, proving that Conjecture (1) holds in this situation. In Section 6, we give a setting where the dual problem is attained locally, showing Conjecture (1) for a case where the marginals are in subharmonic order. 
  In Section~\ref{s: decomposition}, 
we establish the decomposition, as well as the existence of admissible triplets exposing each of the components.   Section~\ref{s: structure} is devoted to proving under various additional conditions, structural results for sets where optimal martingale transports concentrate, while Section~\ref{s:disintegration} deals with the disintegration of martingales along this decomposition and 
 how it is related to the presence of  Nikodym sets.   

The authors are extremely grateful to Luigi Ambrosio for several pertinent remarks and discussions regarding the results of this paper. The first-named author also thanks David Preiss for very helpful discussions and insight.  

\section{Main results}

To discuss our main results, we first introduce  a few definitions. We also borrow some of the notation from \cite{bj}.
\begin{definition}
\label{def1.3}
For $A \subset \R^d$, we shall write $V(A)$ for the lowest-dimensional affine space containing $A$. Also define $IC(A) := int(conv(A))$ and $CC(A) := cl(conv(A))$, where again the interior or closure is taken in the topology of $V(A)$, where the topology of a set $A$ is with respect to the Euclidean metric topology of $V(A)$ (and not with respect to the whole space $\R^d$). 
\end{definition}
If $A = \{x\}$, then $IC(A) = \{x\}$ since we consider the interior of a singleton set is itself in the topology of $0$-dimensional space. 

 In reality, we will be dealing with the vertical fibers  $\Gamma_x =\{ y \in \R^d \ | \ (x, y) \in \Gamma\}$ of a certain class of Borel sets $\Gamma \subset \R^d \times \R^d$, on which martingale measures $\pi \in MT(\mu,\nu)$ would be concentrated. The constraint that $x$ is the barycenter of $\pi_x$, which is normally supported on $\overline \Gamma_x$, 
naturally leads us to the following definition. Recall that for a Borel set $\Gamma \subset \R^d \times \R^d$, 
we write $X_\Gamma :=$ proj$_X \Gamma$, $Y_\Gamma :=$ proj$_Y \Gamma$, i.e. $X_\Gamma$ is the projection of $\Gamma$ on the first coordinate space $\R^d$, and $Y_\Gamma$ on the second. 

\begin{definition} 
\rm
\label{def: martingale supporting}   We say that a Borel  set $\Gamma \subset \R^d \times \R^d$ is a {\em martingale  supporting set,} if
\begin{align}\label{eq: martingale supporting}
 \hbox{ for every $x\in X_\Gamma$, $x \in IC(\Gamma_x)$}.
\end{align}
We let ${\mathcal S}_{MT}$ denote the class of all martingale supporting sets.
\end{definition}

Our first main result shows that martingale supporting sets that are $c$-contact layers enjoy special structural properties. A key step established in Section \ref{Martingale.transform} is to show that an exposing admissible triple can be extended and regularized via a notion of {\it martingale  $c$-Legendre transform}, so that it verifies the needed differentiability properties. 

 \begin{theorem}[\bf Regularization of admissible triples via martingale -Legendre transform]
 \label{when.duality} Let $c$ be a cost function on $\R^d$ such that $x\mapsto c(x, y)$,  resp. $y \mapsto c(x, y)$,  is locally Lipschitz, where the Lipschitz constants are uniformly bounded in $y$ and respectively, in $x$.  
  Let $\Gamma$ be a Borel set in ${\mathcal S}_{MT}$ that is a $c$-contact layer, and suppose that $X_\Gamma \subset \Omega :=  IC(Y_\Gamma)$ with $\Omega$ being an open set in $\R^d$.
  Then, 
   \begin{enumerate}
\item  There exist a locally Lipschitz function $\alpha:  \Omega \to \R$,  a locally bounded $\gamma: \Omega \to \R^d$, and $\beta: \R^d \to \R$,  such that $\Gamma$ is a $c$-contact layer for the triplet  ($\alpha, \gamma, \beta$).
 
 \item   If the admissible triple is in $E_M (c, X_\Gamma, Y_\Gamma)$, and if $y \mapsto c(x, y)$ is assumed to be convex, then $\beta$ can be taken to be a convex function on $\R^d$.  
 
 \item If $c(x,y)=|x-y|$ and the admissible triple is in $E_m (c, X_\Gamma, Y_\Gamma)$, then $ \alpha =\beta$ on $\Omega$. 

   \end{enumerate}
  
\end{theorem}
This will allow us to prove the following structural result.

\begin{theorem}[\bf Extremal structure of a martingale supporting $c$-contact layer]
\label{thm: main extremal}  
Let $c(x,y)=\pm |x-y|$ and assume $\Gamma$ is a $c$-contact layer in ${\mathcal S}_{MT}$.  Then for $\mathcal{L}^d$ - a.e. $x$ in $X_\Gamma$, the  closure $\overline{\Gamma_x}$ of $\Gamma_x$  coincides with the  set of extreme points of the convex hull of  $\overline{\Gamma_x}$, 
i.e., 
$\overline{\Gamma_x} = {\rm Ext} \, \big{(}{\rm conv}(\overline{\Gamma_x})\big{)}$.

In particular, if $\mu$ is a probability measure that is absolutely continuous with respect to the Lebesgue measure, and  if the dual problem is attained, then for any $\pi\in MT(\mu, \nu)$ that is a solution of \eqref{MGTP} for either the minimization or maximization problem,  then  for $\mu$ - a.e. $x$, ${\rm supp}\, \pi_x = {\rm Ext}\,\big{(} {\rm conv}({\rm supp}\, \pi_x) \big{)}$. \end{theorem} 
 We shall see that the dual problem is not always attained. However, a localized version of the above theorem will allow us to deal with a case where the marginals are in subharmonic order. Actually, by letting $P_\mu$ be the Newtonian potential of a probability measure $\mu$, we shall be able to deduce the following result (see Section \ref{S: local dual}).   

\begin{theorem}[\bf Case of marginals in subharmonic order]
Assume $\mu\le_{SH} \nu$ where $\mu, \nu$ are probability measures with compact support on $\R^d$ such that  $\mu <<\mathcal{L}^d$ ($d\geq 3$), and that the open set $\{x \,|\, P_\nu(x)-P_\mu(x) >0\} $ has the full measure of $\mu$.  If $\pi \in M(\mu, \nu)$ is an optimal solution for the problem  \eqref{MGTP}, where  the cost function is $c(x,y) = \pm |x-y|$,  then  for $\mu$ - a.e. $x$, ${\rm supp}\, \pi_x = {\rm Ext}\,\big{(} {\rm conv}({\rm supp}\, \pi_x) \big{)}$.
\end{theorem}

Since martingale supporting sets $\Gamma$ in ${\mathcal S}_{MT}$ are not always $c$-contact layers even when they are concentration sets for optimal martingale transports (\cite{BeHePe11} or Example~\ref{ah} below), we investigate the possibility of decomposing such sets into  ``irreducible components" such that each component becomes a $c$-contact layer.  For that, we introduce the concept of a {\it convex paving}.  
 
 \begin{definition} Let $\Phi$ be a family of mutually disjoint open convex sets in $\R^d$ (Recall that  here, the openness of a set $C$ is with respect to the space $V(C)$). Given a set $\Gamma \subset \R^d\times \R^d$, we shall say that $\Phi$ is a convex paving for $\Gamma$ provided
 \begin{enumerate}
 \item $X_\Gamma \subset \bigcup_{C \in \Phi} C$.
 \item each $C \in \Phi$ contains at least one element $x$ in $X_\Gamma$ ($C$ is then denoted $C(x)$).
  \item For any $z, x \in  X_\Gamma$, we have: $IC(\Gamma_z) \cap C(x) \neq \emptyset \Rightarrow 
 IC(\Gamma_z) \subset C(x)$. 
  \end{enumerate}
 \end{definition}
 \noindent Note that such a paving clearly defines an equivalent relation on $X_\Gamma$ by simply defining $x \sim_\Phi x'$ if and only if $C(x) = C(x')$. The corresponding equivalent classes are then $[x] = C(x) \cap X_\Gamma$.\\
 There can be many convex pavings of a set $\Gamma \subset \R^d\times \R^d$; take for example $\Phi := \{\R^d\}$ which however doesn't give much information about $\Gamma$.  We therefore  
 introduce the following concept.
\begin{definition} For a fixed set $\Gamma \subset \R^d\times \R^d$, we shall say that $\Phi$ is an irreducible convex paving for $\Gamma$ if for any other convex paving $\Psi$ for $\Gamma$, we have the following property: 
If $C \in \Phi$, $D \in \Psi$ are such that $C \cap D \neq \emptyset$, then necessarily $C \subset D$. 
\end{definition}
Note that an irreducible convex paving for a set $\Gamma$ is necessarily unique. 
As to their existence, we shall show in Section~\ref{s: decomposition}
 the following result.
 
 \begin{theorem} 
\label{general.dec.}   \rm 
For every martingale supporting set $\Gamma$ in ${\mathcal S}_{MT}$, there exists a unique irreducible convex paving for $\Gamma$.
 
\end{theorem}
Now, a key property of optimal transport plans in Monge-Kantorovich theory is that they are concentrated on Borel sets that are {\em $c$-cyclically monotone}, which is a property that describes every finite collection of points in the concentration set \cite{Vi03}. Similarly, a key property of an optimal  martingale transport $\pi \in MT(\mu, \nu)$ -- due to Beiglb\"ock and Juillet \cite{bj} -- is a {\it monotonicity property} enjoyed by every finite collection of points in their support. It implies in particular, that there exists a set $\Lambda$ of full $\pi$-measure in $\R^d\times \R^d$  such that  each one of its finite subsets is a $c$-contact layer.
This is one of the consequences of the variational lemma in \cite{bj}, where duality on finite sets is obtained via linear programming (see \cite{bj} and \cite{HoKl12}). 
We therefore introduce the following combinatorial counterpart of cyclic monotonicity for martingale transport.

\begin{definition}\label{exposable} \rm A subset $\Lambda$ of $\R^d\times \R^d$ is said to be {\em $c$-finitely exposable} for some cost function $c$, if each one of its finite subsets is a $c$-contact layer.
\end{definition} 

The following proposition describes the combinatorial nature of the support of of optimal martingale transports. 

\begin{proposition}\label{pp}
Let  $\pi \in MT(\mu, \nu)$ be an optimal martingale transport for Problem \eqref{MGTP}. Assuming the cost $c$ continuous, then there exists a $c$-finitely exposable concentration set $\Lambda$ for $\pi$. 
\end{proposition} 
Indeed, it is shown in  \cite{bj} (see also \cite{Z}) that there exists a Borel set $\Lambda$ in $\R^d \times \R^d$ 
with $\pi(\Lambda)=1$, that satisfies a certain monotonicity property, which is  the martingale counterpart of the $c$-cyclic monotonicity that is inherent to the Monge-Kantorovich theory. As mentioned above, by the duality theorem of linear programming, this property is equivalent to saying that every finite subset of $\Lambda$ is a $c$-contact layer. 


Since duality is not attained in general, an optimal martingale transport measure is  not necessarily concentrated on a $c$-contact layer $\Gamma \in \mathcal{S}_{MT}$. On the other hand,  we can and will assume that it is concentrated on a set $\Gamma \in \mathcal{S}_{MT}$ whose finite subsets are $c$-contact layers. This leads to the question of finding ``maximal" components of $\Gamma$ that are $c$-contact layers. It turns out that this is indeed the case as we show that 
$\Gamma_C:=\Gamma \cap (C \times \R^d)$ is a $c$-contact layer for any component $C$ of the irreducible convex paving $\Phi$ of $\Gamma$. It is summarized in the following theorem.

\begin{theorem} 
\label{th:k} Let  $\Gamma$ be a $c$-finitely exposable set in ${\mathcal S}_{MT}$, 
then there exists an irreducible convex paving $\Phi$ for  $\Gamma$ such that for every convex component $C$ in $\Phi$,
the set $\Gamma \cap (C\times \R^d)$ is a $c$-contact layer. 
\end{theorem}
\begin{remark} \label{rm: compare with Rockafellar} 
  Theorem~\ref{th:k} can be seen as a martingale counterpart to a celebrated result of Rockafellar \cite{Rock} in the Monge-Kantorovich theory for mass transport, which essentially says that the property of $c$-cyclical monotonicity that characterizes the support of optimal transport plans are somewhat ``$c$-contact layers" exposed by a pair of functions, one being $c$-convex and the other being its $c$-Legendre transform. Here, {\it $c$-finite exposability} replaces  {\it $c$-cyclic monotonicity}, while ``exposing" martingale supporting sets require a new notion of duality between a function $\beta$ and a pair of functions $(\alpha, \gamma)$.
  However, in the martingale case, the whole support is not necessarily a $c$-contact layer, but every irreducible component is.  
  \end{remark}
\noindent Theorems~\ref{when.duality} and \ref{th:k} yield several structural results in general dimensions such as the following.
Note that the attainability of the dual problem is not assumed here.
 \begin{corollary}[\bf dimensional result]\label{cor: dim d-1}
Let $\pi$ be a solution of the optimization problem \eqref{MGTP} with $c(x,y) = |x - y|$ and suppose $\mu$ is absolutely continuous with respect to the Lebesgue measure. Then for $\mu$-almost every $x$ in $\R^d$, 
\begin{enumerate}[(1)] 
 \item the Hausdorff dimension of  $\supp \pi_x$ is at most $d-1$, and
 \item If $dim\,V(\supp \pi_x)=d$, then \, ${\rm supp}\, \pi_x = {\rm Ext}\,\big{(} {\rm conv}({\rm supp}\, \pi_x) \big{)}$. 
  \end{enumerate} 
\end{corollary}
\begin{proof}
Indeed, there exists  $\Gamma \in  \mathcal{S}_{MT}$ with $\pi(\Gamma)=1$ that is $c$-finitely exposable, and such that  $\overline{\Gamma_x} = \supp \pi_x$ for $\mu$ a.e. $x$ (See Appendix A). Now, consider those points $x$ with $\dim V(\Gamma_x) =d$. 
 In this case, the disjoint sets $C(x)$ in Theorem~\ref{th:k} are open sets in $\R^d$ and so, the restriction of $\mu$ to each of the components is again absolutely continuous.  Theorems~\ref{when.duality} and  \ref{thm: main extremal} can then be applied. Note now that the  the set of extreme points has  dimension at most $d-1$. This shows that for $\mu$-a.e. $x$ in the open set $\bigcup_{\dim V(C)=d}C$, we have that  $\dim (\overline{\Gamma_x}) \le d-1$. The property also obviously holds outside that set, which means that item (1) is also verified. 
\end{proof}
A more involved application of the decomposition is a complete solution of Conjecture 1) in two dimensions, namely the following, which is proved in Section 8. 

\begin{theorem} 
\label{th: Choquet 2d.0}
Assume $d=2$, $c(x,y) = |x - y|$, $\mu$ is absolutely continuous with respect to the Lebesgue measure, and  $\nu$ has compact support. Let $\pi\in MT(\mu, \nu)$ be a solution of \eqref{MGTP},  then  for $\mu$ - a.e. $x$, ${\rm supp}\, \pi_x = {\rm Ext}\,\big{(} \overline{{\rm conv}}({\rm supp}\, \pi_x) \big{)}$.
\end{theorem} 
The decomposition also allows us to  give  in Section 8  the following positive answer to Conjecture 2, whenever  the target measure is discrete. Note that in this case, the result holds true in both the maximization and minimization problems. 

 \begin{theorem} \label{th: finite images}
 Let $c(x,y) = |x - y|$  (or more generally for $c(x,y)= |x-y|^p$ with $p\ne 2$), suppose $\mu$ is absolutely continuous with respect to the Lebesgue measure, and that $\nu$ is discrete;  i.e. $\nu$ is supported on a countable set. If $\pi \in MT(\mu, \nu)$ is an optimizer for \eqref{MGTP}, then   for $\mu$ - a.e. $x$, $\supp \pi_x$ consists of exactly $d+1$ points which are vertices of a $d$-dimensional polytope in $\R^d$, and therefore the optimal solution is unique.
\end{theorem}
Now we give a couple of examples, which illustrate that the above stated conjectures could be the best structural results we can hope for. 
\begin{example}
\label{rm: no finite for max}
  The polytope-like structure of the support required in Conjecture 2 
  does not hold in general for the corresponding maximization problem.  
 Indeed, since 
$\frac{1}{2} (|x-y| - 1)^2 \geq 0$, we have
\begin{align}\label{u}
\frac{1}{2} |y|^2 - \frac{1}{2}|x|^2 + 1- x \cdot (y-x) \geq |x-y| \quad \hbox{on $\R^d \times \R^d$,} 
\end{align}
with equality on the set $\{(x,y); |x-y|=1\}$. The functions $\alpha (x) = \frac{1}{2}|x|^2 - 1$, $\beta(y) = \frac{1}{2} |y|^2 $ and $\gamma (x) = x$ then form a dual triplet for the maximization problem with cost $|x-y|$. This means that every martingale $(X, Y)$ with $|X-Y|=1$ a.s. is optimal for the maximization problem corresponding to its own marginals $X \sim \mu$ and $Y \sim \nu$.
 Hence, $\supp \pi_x$  is not in general a discrete set,  and indeed, $\supp \pi_x$ can attain the Hausdorff dimension $d-1$.
\end{example}
We now consider the uniqueness question in Conjecture 2, and whether it could hold for the maximization problem. In \cite{bj} it is shown  that when $d=1$ the solution of the martingale transport problem \eqref{MGTP} is unique for both max/min problem under the assumption that $\mu$ is absolutely continuous. Also, it is reported in \cite{Lim} that in the minimization problem with radially symmetric marginals $(\mu, \nu)$, the minimizer is again unique in any dimension. We note however that, unlike the minimization case, one cannot expect the uniqueness of a maximizing martingale measure in higher dimensions, even in the radially symmetric case, as the following example indicates.

\begin{example}
\label{ex: nonunique max} Let $\mu$ be a radially symmetric probability measure on $\R^2 \simeq   \C$ such that $\mu(\{0\}) = 0$. Let $z_1= \cos \frac{\pi}{4} +i \sin \frac{\pi}{4},\,  z_2 = \cos \frac{\pi}{4} - i \sin \frac{\pi}{4}, \, z_3 = - z_1$ and \, $z_4 = - z_2$, and define the probability measures $\pi_1$ and $\pi_2$ on $\C \times \C$, whose disintegrations $\pi^1_x$ and $\pi^2_x$ for each $x \in \C, x \neq 0,$  are given by,
\begin{align*}
\pi^1_x = \frac{1}{4} \delta_{x + \frac{x}{|x|} z_1} + \frac{1}{4} \delta_{x + \frac{x}{|x|} z_2} + \frac{1}{4}\delta_{x + \frac{x}{|x|} z_3} + \frac{1}{4}\delta_{x + \frac{x}{|x|} z_4} \\
\pi^2_x = \frac{1}{8} \delta_{x + \frac{x}{|x|} z_1} + \frac{3}{8} \delta_{x + \frac{x}{|x|} z_2} + \frac{1}{8}\delta_{x + \frac{x}{|x|} z_3} + \frac{3}{8}\delta_{x + \frac{x}{|x|} z_4}.
\end{align*}
Then, by the discussion in Example~\ref{rm: no finite for max}, one can see that both $\pi_1$ and $\pi_2$ are optimal for the maximization problem corresponding to $\mu$ and $\nu:=\nu_1=\nu_2$, where $d \nu_i (y) = \int_{\C} \pi^i_x (y)\, d\mu(x)$, $i=1,2$, hence, the maximizer is not unique.
\end{example}

Finally, we consider in Section 9 whether  one can perform a disintegration of $\pi$ with respect to the decomposition $\{ \Gamma_C\}_{C\in \Phi}$ into components $(\pi_C)_C$ in such a way that each $\pi_C$ is a  probability measure supported on $\Gamma_C:=\Gamma \cap (C \times \R^d)$ and $\pi_C\in MT(\mu_C, \nu_C)$, where $\mu_C, \nu_C$ are suitable 
probability measures  in convex order, with  $\mu_C$ is supported on  $X_C : = X_\Gamma \cap C$ and $\nu_C$ on $Y_{\Gamma_C}$. The advantage of this decomposition is that  If $\pi$ is optimal for  problem (\ref{MGTP}) in MT($\mu, \nu$), then $\pi_C$ is optimal for the same problem on $MT(\mu_C, \nu_C)$, with the added property that $\Gamma_C$ is a $c$-contact layer, which means that duality is attained for each $\pi_C$. 
The decomposition of $\Gamma$ given by 
Theorem~\ref{th:disintegration} 
was motivated by a similar one proposed by Beiglb\"ock-Juillet \cite{bj} in the one dimensional case $(d=1)$.  Our decomposition is however quite different since it depends on the concentration set $\Gamma$ for $\pi$, while in their case the decomposition depends only on the marginals $\mu$ and $\nu$. Theirs is also a countable partition, which makes the restricted problems much more amenable to analysis.   
Actually, the intervals in their decomposition are simply the connected components of the set where the potentials of $\mu$ and $\nu$ are different on the real line. However, in the higher dimensional cases our decomposition can be uncountable, and that's why we talk about a disintegration as  opposed to a decomposition. Moreover, the induced probability measures $\mu_C$'s can be Dirac measures (see Example \ref{nikodym}), which means that Theorem \ref{thm: main extremal} may not be applicable to each piece $\pi_C$ even if duality is attained for the restricted problem. We refer to Section 9 for the challenges and the interesting questions arising from this fundamental decomposition.
 

\section{The martingale  $c$-Legendre transform}\label{Martingale.transform}

In this section, we investigate properties of the admissible triplet of functions that appear in the dual martingale problem and their associated contact layers. Note that in the case of standard mass transport problems, the contact layer is determined by a potential function and its $c$-Legendre transform, whose regularity properties are inherited from those of $c$, and which can be studied independently of the primal transport problem. A similar methodology works in our setiing, once we introduce an appropriate Legendre duality.

\begin{definition} \rm Let $Y$ be a Borel set in $\R^d$ such that $\Omega := IC(Y)$ is open in $\R^d$, and let $\beta : Y \to \R$ be a Borel function such that for some $s \in \R, t \in \R^d, x_0 \in \Omega$, we have
\begin{align}\label{bound}
\beta(y) \leq c(x_0, y) +t \cdot (y - x_0)  + s \quad \hbox{for all $y \in Y$.}
\end{align}
\begin{enumerate}
\item The {\it martingale  $c$-Legendre dual of the function $\beta$ on $\Omega$} is the pair $\beta_c:=(\alpha_c, \gamma_c)$, where $\alpha_c:\Omega \to \R$ is given by 
\begin{align}\label{eq: bar alpha.1}
\alpha_c (x) := \inf \{a \in \R : \exists b \in \R^d \,\,\text{ such that }\,\,\beta(y) - c(x, y) \le b \cdot (y-x) + a, \,\,\forall y\in Y\}, 
\end{align}
and $\gamma_c: \Omega \to \R^d$ is the possibly set-valued function defined by 
\begin{align}\label{eq: bar gamma.1}
\gamma_c(x) := \{b \in \R^d : \beta(y)- c(x,  y) &\le b \cdot (y-x) + \alpha_c(x), \,\, \forall y \in Y \}.  
\end{align}

\item The {\it martingale  $c$-Legendre dual of a pair of functions $(\alpha, \gamma):\Omega \to \R\times \R^d$} is the function $(\alpha, \gamma)_{_{c}}: \R^d \to \R$ defined by
\begin{align}\label{double.transform}
(\alpha, \gamma)_{_{c}} (y) :=
\inf_{x\in \Omega, b \in \gamma (x) } \{ c(x,  y) +  b   \cdot (y-x) +  \alpha (x)\}.   
\end{align}
\item We shall denote by $\beta_{cc}$ the martingale  $c$-Legendre dual of the pair  $\beta_c=(\alpha_c, \gamma_c)$, and say that $\beta$ is {\it martingale  $c$-convex on $Y$}, if $\beta=\beta_{cc}$ on $Y$. 
\end{enumerate}
\end{definition}
In order to emphasize the analogy with the standard Fenchel-Legendre duality, we shall write 
\[
\beta_c (x, y)=(\alpha_c, \gamma_c)(x,y):=\alpha_c(x) +\gamma_c(x)(y-x).
\]

\begin{theorem}\label{thm: alpha beta difference} Assume that $(x,y) \mapsto c(x,y)$ is continuous and   $x \mapsto c(x, y)$ (respectively $y\mapsto c(x,y)$) is locally Lipschitz with local Lipschitz constants  uniformly bounded in $y$ (respectively in $x$). 
Let $Y$ be a Borel set in $\R^d$ such that $\Omega := IC(Y)$ is open in $\R^d$, and let $\beta : Y \to \R$ be a Borel function satisfying (\ref{bound}), and $\beta_c=(\alpha_c, \gamma_c)$ its martingale  $c$-Legendre dual. Then 
\begin{enumerate}

\item $\alpha_c$ is  locally Lipschitz  in $\Omega$, while $\gamma_c$ and $\beta_{cc}$ are locally bounded in $\Omega$.

\item $\beta \leq \beta_{cc}$ on $Y$, and 
\begin{equation}
\hbox{$\beta_{cc}(y)-\beta_c (x, y)\leq c(x,y)$\quad for all  $(x, y)\in  \Omega \times \R^d$.}
\end{equation}
In other words, the triple $(\beta_c, \beta_{cc})=(\alpha_c, \gamma_c, \beta_{cc}) \in E_m (c,  \Omega, \R^d)$.
\item   
 $ \beta_{cc} (x) -\delta_1 \le  \alpha_c (x)  \le \beta_{cc}(x) +  \delta_2$ for all $x \in \Omega$, where
 \[
\delta_1 = \sup_{x \in \Omega} c(x, x)\quad {\rm and} \quad  \delta_2  = \sup_{x, x' \in \Omega,  y \in Y } [ c(x,y) - c(x, x')  - c(x', y)].
\]
 \item  Let $X \subset \Omega$ and let $( \alpha, \gamma)$ be defined on $X$ such that 
 $( \alpha, \gamma,\, \beta)\in E_m (c,  X, Y)$, then $\alpha(x) \ge \alpha_c (x)$ on $X$.
 Moreover, if a $c$-contact layer $\Gamma \subset \Gamma_{(\alpha, \gamma, \beta)}$
 belongs to ${\mathcal S}_{MT}$, then
  \[
 \hbox{$\alpha(x) = \alpha_c (x)$ \,and \,$\gamma(x) \subset \gamma_c(x)$ \,on $X_\Gamma$, \,$\beta_{cc} = \beta$ on $Y_\Gamma$, \,and \,$\Gamma \subset \Gamma_{(\alpha_c, \gamma_c,  \beta_{cc})}.$}
 \] 
   
\label{cccc} \item The function $\beta_{cc}$ is martingale  $c$-convex on $\R^d$, that is $\beta_{cc} = \beta_{cccc}$ \,\,on $\R^d$.\label{cccc}
 
\end{enumerate}

\end{theorem}
\begin{proof} (1) We first show  that $\alpha_c$ is locally bounded in $\Omega$.  For $x\in\Omega = IC(Y)$, we may choose $\{y_1,...,y_s\} \subset Y$ such that
\begin{align}\label{ag}
x \in U := IC(\{y_1,...,y_s\})\, \text{ and }\,  \hbox{$U$ is open in  $\R^d$.}  
\end{align}
\noindent Since $x =\sum_i t_i y_i$, $\sum_i t_i =1$, $t_i \ge 0$, 
 it is clear that    $\alpha_c(z) \ge M(z):=   \min_{y_i} [\beta(y_i) - c (z, y_i)]$
for all $z \in U$.  In view of the continuity of $c$, this yields that $\alpha_c$ is locally lower  bounded. \\
We now prove  that  $\alpha_c$ is locally upper bounded. Indeed, fix $R>0$ and let $x \in \Omega, y \in Y$ be such that $|x_0|, |x| <R$. 
By the { local Lipschitz property of $c$ in $x$, i.e.
\begin{align*}
 |c(x, y) - c(x_0, y)| \le C|x-x_0|
\end{align*}
for some $C=C(R)>0$ 
and for all $|x| < R$, 
we have that
}
\begin{align}\label{dual3}
s + t  \cdot (y-x_0) \ge \beta(y) - c(x_0, y) \ge  \beta(y)  - c(x , y) -  C|x-x_0|.
\end{align}
Thus,
\begin{align*}
s + C|x-x_0|+t \cdot (x-x_0) + t \cdot (y-x) \ge \beta(y) - c(x, y). 
\end{align*}
The definition of   $\alpha_c$ gives
\begin{align}\label{alpha}
s +   C |x-x_0|+t \cdot (x-x_0) \ge  \alpha_c (x).
\end{align}
In particular, $\alpha_c$ is locally upper bounded, hence 
locally bounded. \\
Note now that $\gamma_c(x)$ is a set valued  function, and it is clearly closed and convex for each $x \in \Omega$. To see the local boundedness of $\gamma_c$, use \eqref{ag} and let $V$ be a small neighbourhood of $x$ whose closure is in $U$. Since $\alpha_c$ is bounded on $V$, there exists a constant $C$ such that  
\begin{align}
b \cdot (y_i-z) \geq C, \quad  \,\,\forall z \in V,\, i=1,2,...,s, \quad \forall b \in \bar \gamma(z) 
\end{align}
which says that $\gamma_c$ is bounded on $V$, thus locally bounded on $\Omega$.  To show that $\gamma_c(x)$ is nonempty  for any $x\in \Omega$, choose an approximating sequence $\{a_n\}\subset \R$ for $\alpha_c(x)$ and corresponding $\{b_n\} \subset \R^d$, in such a way that  
$\beta(y) - c(x, y) \le b_n \cdot (y-x) + a_n$ and $a_n \searrow \alpha_c(x)$. Now the above argument shows that $\{b_n\}$ must be bounded, hence its accumulation points must be in $\gamma_c(x)$. \\
We now show that $\alpha_c$ is locally Lipschitz. Since $\alpha_c$ is finite in $\Omega$, the above argument showing the local boundedness for $\alpha_c$ can be repeated, giving \eqref{alpha} for any $x , x_0 \in \Omega$, $s= \alpha_c(x_0)$ and $ t\in \gamma_c(x_0)$;
\begin{align*}
 \alpha_c(x_0) +  C |x-x_0|+\gamma_c(x_0) \cdot (x-x_0) \ge  \alpha_c(x).
\end{align*}
 By interchanging $x$ and $x_0$, we get 
\begin{align*}
| \alpha_c(x) - \alpha_c(x_0) |  \leq \left((| \gamma_c (x)| \vee |\gamma_c(x_0)|) +  C \right) |x-x_0|.
\end{align*}
Therefore, the local boundedness of $\gamma_c$ implies that  $\alpha_c$ is locally Lipschitz in $\Omega$. If furthermore $x\mapsto c(x, y)$ is Lipschitz (with Lipschitz constant uniformly in $y$) and $\gamma_c$ is bounded, then the above estimate shows that  $\alpha_c$ is Lipschitz in $\Omega$. \\
As for $\beta_{cc}$, it is clear that it is measurable and  locally upper bounded. It is also clear that 
\begin{align}
(\alpha_c, \gamma_c,  \beta_{cc})\in E_m (c,  \Omega, \R^d) \text{ \,and \,} \beta \le \beta_{cc} \text{ on } Y.
\end{align}
We now show that $\beta_{cc}$ is locally bounded In $\Omega$, by following a similar argument as for $\alpha_c$. 
First,  let $ x \in \Omega, y\in Y,  y'\in \Omega$. By the local Lipschitz property of $c$ in $y$, i.e.
\begin{align*}
 |c(x, y) - c(x, y')| \le C|y-y'|,
\end{align*}
for some $C=C(R)>0$,
and for all $|y|, |y'| < R$, 
we see
 \begin{align}\label{eq: beta and} \nonumber
 \beta(y) &\le  c(x,  y) +  \gamma_c(x) \cdot (y-x) +  \alpha_c(x) \\
 & \le c(x, y') + \gamma_c(x) \cdot (y'-x) + \alpha_c(x) + \gamma_c(x)  \cdot(y-y')+ C|y-y'|.
\end{align}
Now, since $y' \in \Omega = IC(Y)$, one can choose $\{ y_1, ..., y_s\} \subset Y$ such that 
\begin{align*}
 y' \in W= IC(\{ y_1, ..., y_s\}) \quad \hbox{ and $W$ is open in $\R^d$.}
\end{align*}
Thus \eqref{eq: beta and} implies, after putting $y_i$'s in place of $y$ and summing up with appropriate weights, that 
\begin{align*}
\min_{y_i} \beta (y_i) \le \beta_{cc}(y') + C\max_{y_i}|y_i-y'|, 
\end{align*}
hence yielding the local lower boundedness, thus the local boundedness of $\beta_{cc}$ in $\Omega$. This completes the proof of  the items (1) and (2). 

  In order to establish (3), we first note that the inequality $\beta_{cc}  - \delta_1 \le \alpha_c$ on $\Omega$ follows from the fact that $(\alpha_c, \gamma_c, \beta_{cc})\in E_m (c,  \Omega, \R^d)$. For the other inequality,  notice that for each $x \in \Omega$ and an arbitrary $\epsilon >0$, 
there is $x_\epsilon \in \Omega$ and $b \in  \gamma_c(x_\epsilon)$ (which we will simply write as $ \gamma_c(x_\epsilon)$ in the sequel), such that 
\begin{align}\label{eq: beta epsilon bound}
    \alpha_c(x_\epsilon) +  \gamma_c(x_\epsilon)\cdot (x-x_\epsilon) + c(x_\epsilon, x) -\epsilon \le  \beta_{cc} (x).
\end{align}
Let $a_\epsilon (x) := \alpha_c(x_\epsilon) +   \gamma_c(x_\epsilon)\cdot (x-x_\epsilon) + c(x_\epsilon, x)$, and 
consider for $z\in Y$, the function
\begin{align*}
 L(z) =  a_\epsilon (x) +  \gamma_c (x_\epsilon) (z - x) + c(x, z).
\end{align*}
Then,
\begin{align*}
  \beta_{cc}(z) - L(z) & \le  \alpha_c(x_\epsilon) +  \gamma_c(x_\epsilon)\cdot (z-x_\epsilon) + c(x_\epsilon, z)\\
 & \ \ \ 
 -(\alpha_c(x_\epsilon) + \gamma_c(x_\epsilon)\cdot (x-x_\epsilon) + c(x_\epsilon, x))
 - \gamma_c(x_\epsilon) (z-x) - c(x, z)\\
 & \le c(x_\epsilon, z) - c(x_\epsilon, x) -c(x,z) \le \delta_2.
\end{align*}
Hence $ \beta_{cc} (z) \le L(z) +\delta_2$, and therefore $\beta(z) \le L(z) + \delta_2$ for $z \in Y$ by item (1).   From the definition of $ \alpha_c$, this implies
$  \alpha_c (x) \le a_\epsilon (x) + \delta_2$, and 
 from \eqref{eq: beta epsilon bound}, we have
$  \alpha_c (x) \le  \beta_{cc} (x) +\epsilon +\delta_2. $
%
 Since $\epsilon$ is arbitrary, the proof of (3) is complete. 
 
To prove (4), first note that if $X \subset \Omega$ and $( \alpha, \gamma, \beta)\in E_m (c,  X, Y)$, then the definition of $\alpha_c$ obviously implies that $\alpha \ge \alpha_c$ on $X$. Now assume that $\Gamma \subset \Gamma_{(\alpha, \gamma, \beta)}, \Gamma \in S_{MT}$ and in particular, for each $x \in X_\Gamma$, $x \in IC(\Gamma_x)$. Let $x =\sum_i t_i y_i$, $\sum_i t_i =1$, $t_i \ge 0$, $y_i \in \Gamma_x$, 
and observe that
\begin{align}
\label{qw1}\beta(y_i)- c(x, y_i)  &= \gamma(x) \cdot (y_i-x) + \alpha(x)\\
\label{qw2}\beta(y_i)- c(x, y_i)  & \le  \gamma_c(x) \cdot (y_i-x) +  \alpha_c(x)
\end{align}
where the first identity is due to the definition of $\Gamma_{(\alpha, \gamma, \beta)}$ and the second inequality is due to the definition of $\beta_c=(\alpha_c, \gamma_c)$. 
Summing up the above relations with the weights $t_i$, we get 
\begin{align*}
 \alpha(x) = \sum_i t_i \left(\beta(y_i)- c(x, y_i) \right)\le \alpha_c(x). 
\end{align*}
\noindent As $\alpha_c \le \alpha$ on $X$, this shows $\alpha(x) = \alpha_c(x)$ on $X_\Gamma$ and hence $\gamma(x) \subset \gamma_c(x)$ on $X_\Gamma$. Then for $x \in X_\Gamma$, by subtracting \eqref{qw1} from \eqref{qw2}, we get $(\gamma_c(x) - \gamma(x)) \cdot (y-x) \ge 0 \text{ for all } y \in \Gamma_x$. But since $x \in IC(\Gamma_x)$, this implies
\begin{align*}
(\gamma_c(x) - \gamma(x)) \cdot (y-x) = 0 \text{ for all } y \in \Gamma_x.
\end{align*}
In other words, the projection of $\gamma(x)$ and $\gamma_c(x)$ onto the affine subspace generated by $\Gamma_x$ are equal. Now note that \eqref{qw2} obviously holds for $\beta_{cc}$ in place of $\beta$. Again by subtraction, we get $\beta_{cc}(y) \le \beta(y)$ for all $y \in \Gamma_x$. As the reverse inequality is already shown, 
 we see that $\beta=\beta_{cc}$ on $Y_\Gamma$. Moreover, if $(x,y) \in \Gamma$, in other words if $(x,y)$ satisfies \eqref{qw1}, then the above discussion implies that \eqref{qw1} holds with $(\alpha_c, \gamma_c,\, \beta_{cc})$. In other words, $(x,y) \in \Gamma_{(\alpha, \gamma_c, \, \beta_{cc})}$. \\
For item (5), we first note that $\beta_{cc}$ is defined on $\R^d$ and by item (2), we have $\beta_{cc} \leq \beta_{cccc}$. For the reverse inequality, fix $z \in \R^d$. Then by definition of $\beta_{cc}$, there exist a sequence $\{x_n\}$ in $\Omega$ and $b_n \in \gamma_c (x_n)$, $n\ge 1$, such that
\begin{align*}
 \beta_{cc} (y) &\le c(x_n,  y) +  b_n   \cdot (y-x_n) +  \alpha_c (x_n) \quad \text{for every $y \in \R^d$, and}\\
  \beta_{cc} (z) &=\lim_{n \to \infty} c(x_n,  z) +  b_n   \cdot (z-x_n) +  \alpha_c (x_n).
 \end{align*}
 This readily implies that $\beta_{cc} (z) = \beta_{cccc} (z)$, completing the proof of the theorem.
 
\end{proof}
\begin{remark} Note that both costs $c(x,y) = |x - y|$ and $c(x,y) = - |x - y|$ satisfy the above hypothesis, and in both cases, i.e., $c(x,y) =\pm|x-y|$, we have that $\delta_1 =0$. Moreover, $\delta_2 \le 2 {\rm diam} (\Omega)$ if $c(x,y) =- |x-y|$. \\
On the other hand, if $c(x,y) = |x-y|$,  then $\delta_2 =0$, which means that $\alpha_c =\beta_{cc}$ on $\Omega$.    In particular, by Theorem \ref{thm: alpha beta difference} \eqref{cccc}, the duality theorem becomes
  \begin{eqnarray*}\label{No.gap.10}
\min\left\{ \int_{\R^d\times \R^d} |x-y| \,d\pi \, ;\, \pi\in MT(\mu,\nu)\right\} 
=\sup \left\{\int_{\R^d}\beta \, d(\nu -\mu)\, ; \, \hbox{$\beta$ is martingale  $c$-convex}
\right\},
\end{eqnarray*}
which can be seen as the counterpart of the Kantorovich-Rubenstein duality formulation in standard transport theory, whenever the cost is given by a distance function. 
\end{remark}

\begin{remark}\label{localization} {\bf(Localization)} Let $K$ be a compact set in $\Omega$ and let $\alpha_c^K$ and $\gamma_c^K$ be the restrictions of $\alpha_c$ and $\gamma_c$ on $K$, then $(\alpha_c^K, \gamma_c^K,\, \beta_{cc}^K)\in E_m (c,  K, \R^d)$, where 
$$\beta_{cc}^K (y) := \inf_{x\in K } \{ c(x,  y) +  \gamma_c^K(x)   \cdot (y-x) +  \alpha_c^K(x)\}.$$
Consequently, $\alpha_c^K$ is  Lipschitz  in $K$, and $\gamma_c^K$ is  bounded in $K$. Moreover,  
$\beta_{cc}^K$ is Lipschitz (resp., locally Lipschitz) in $\R^d$ provided $y \mapsto c(x,y)$ is Lipschitz (resp., locally Lipschitz) in $\R^d$. 

Indeed, from the definition of $\beta_{cc}^K$, the boundedness of $\gamma_c$ on $K$ and the local Lipschitz  assumption on $y\mapsto c(x, y)$ (uniformly in $x$), we see that $\beta_{cc}^K$ is the infimum of local Lipschitz functions parametrized by $x \in K$ with the local Lipschitz constant uniform in $x$. This shows that $\beta_{cc}^K$ is locally Lipschitz in $\R^d$. If in addition, $c$ is Lipschitz, then by the same reasoning $\beta_{cc}^K$ is Lipschitz in $\R^d$. 

\end{remark} 

\section{Extremal structure of a $c$-contact layer}\label{triplet}

We first deal with the differentiability properties of an admissible triple $(\alpha, \gamma, \beta)$.
The next lemma shows that essentially $\gamma$ is differentiable in an appropriate sense, wherever $\alpha$ is. This property will be crucial in the proof of Theorem~\ref{thm: main extremal}.

 \begin{lemma}
 \label{lem: alpha gamma}
Suppose $x \mapsto c(x,y)$ is differentiable at $x$ whenever $x \neq y$, and  
assume that $\Gamma$ is a set in ${\mathcal S}_{MT}$ that is a $c$-contact layer for a triple $(\alpha, \gamma, \beta) \in E_m(c, \Omega, \R^d)$, where ${\alpha} : \Omega \to \R$, ${\beta} : \R^d \to \R$, and ${\gamma} : \Omega \to \R^d$. Fix $x\in X_\Gamma$, and let $V$ be the vector subspace of $\R^d$ corresponding to the affine space $V(\Gamma_x)$, and assume $dim(V) \geq 1$. 
Assume there is $s \in V$ such that 
\begin{align}\label{lower}
\alpha (x') \   \le  \ 
s \cdot (x'-x) + \alpha(x) +o(|x'-x|)\, \text{ as }\, x'  \rightarrow x \, \text{ in }\, V(\Gamma_x) .
\end{align}
  Let $\text{proj}_V \gamma$ be the orthogonal projection of the value of $\gamma$ on  $V$.  Then $\alpha$ and $\text{proj}_V \gamma$ have a directional derivative at $x$ in every direction $u \in V$.  \end{lemma}
\begin{proof}
 By the duality assumption 	for the  minimization problem, for all $x' \in \Omega$ and all $(x,y) \in \Gamma$,  
\begin{align}\label{dual11} 
c(x',y) + \gamma(x') \cdot (y-x') +\alpha(x')\   \geq  \  
 c(x,y) + \gamma(x) \cdot (y-x) +\alpha(x).
\end{align}
 Choose a unit vector $u \in V$ and let $x' = x+tu$. Then \eqref{dual11} is rewritten as
\begin{align}\label{dual2a} 
\frac{\alpha(x+tu) - \alpha(x)}{t}  \ge  
\frac{\gamma(x+tu) - \gamma(x)}{t} \cdot (x+tu-y) + \gamma(x)\cdot u - \frac{c(x+tu,y)-c(x,y)}{t} \,\,\text{if}\,\, t>0\\
\label{dual2b} \frac{\alpha(x+tu) - \alpha(x)}{t}  \le  
\frac{\gamma(x+tu) - \gamma(x)}{t} \cdot (x+tu-y) + \gamma(x)\cdot u - \frac{c(x+tu,y)-c(x,y)}{t} \,\,\text{if}\,\, t<0
\end{align}
Let us use the notation $D_{t,u}f(x) = \frac{f(x+tu) - f(x)}{t}$. 
Now the assumption \eqref{lower} says that
\begin{align}\label{af}
\limsup_{t \downarrow 0} D_{t,u} \alpha(x)  \le s \cdot u \le      \liminf_{t \uparrow 0} D_{t,u} \alpha(x).
\end{align}
\noindent Since $x \in int(conv(\Gamma_x))$,  there exists $y_1,...,y_k \in \Gamma_x \setminus \{x\}$, $p_1,...,p_k \geq 0$, $q_1,...,q_k \geq 0$, $\Sigma p_i = 1$, $\Sigma q_i = 1$, $t_+ > 0$, $t_- <0$, such that 
\begin{align*}
x + t_+u = \Sigma p_i y_i\\
x + t_-u = \Sigma q_i y_i.
\end{align*}
Note that the first term on the right side of \eqref{dual2a} and \eqref{dual2b} is linear in $y$, so by summing up the $y_i$'s with the  weights $p_i$'s or $q_i$'s, we get (and we write $\gamma_1(x) := \gamma(x) \cdot u$)
\begin{align}\label{dual3a} 
D_{t,u}\alpha(x) \ge D_{t,u}\gamma_1(x) (t - t_\pm) + C_\pm(t)  \,\,\text{if}\,\, t>0\\
\label{dual3b} D_{t,u}\alpha(x) \le D_{t,u}\gamma_1(x) (t - t_\pm) + C_\pm(t)   \,\,\text{if}\,\, t<0
\end{align}
Here $C_+(t), C_-(t)$ are functions of $t\neq0$, but have limits as $t \to 0$ by the differentiability assumption on the cost. Write $C_\pm = \lim_{t\rightarrow0} C_\pm(t)$, respectively.\\
By taking $\limsup_{t\downarrow 0}$ in \eqref{dual3a} and $\liminf_{t\uparrow 0}$ in \eqref{dual3b} and by recalling  that $t_+>0, t_-<0$, we have
\begin{align*}
\limsup_{t\downarrow 0} D_{t,u}\alpha(x) &\ge (- t_+)  \liminf_{t\downarrow 0} D_{t,u}\gamma_1(x) + C_+ \\
\limsup_{t\downarrow 0} D_{t,u}\alpha(x) &\ge (- t_-)  \limsup_{t\downarrow 0} D_{t,u}\gamma_1(x) + C_- \\
\liminf_{t\uparrow 0} D_{t,u}\alpha(x) &\le (- t_+)  \limsup_{t\uparrow 0} D_{t,u}\gamma_1(x) + C_+ \\
\liminf_{t\uparrow 0} D_{t,u}\alpha(x) &\le (- t_-)  \liminf_{t\uparrow 0} D_{t,u}\gamma_1(x) + C_-.
\end{align*}
\noindent This and \eqref{af} combine to give 
\begin{align*}
\liminf_{t\uparrow 0} D_{t,u}\gamma_1(x) &\ge \limsup_{t\downarrow 0} D_{t,u}\gamma_1(x) \ge
\liminf_{t\downarrow 0} D_{t,u}\gamma_1(x)\\ 
&\ge \limsup_{t\uparrow 0} D_{t,u}\gamma_1(x)
\ge \liminf_{t\uparrow 0} D_{t,u}\gamma_1(x), 
\end{align*}
that is $\gamma_1 = \gamma \cdot u$ is differentiable at $x$ in the direction $u$. Knowing this, we then take  
$\liminf_{t\downarrow 0}$ in \eqref{dual3a} and $\limsup_{t\uparrow 0}$ on \eqref{dual3b} 
to get
\begin{align*}
\liminf_{t\downarrow 0} D_{t,u}\alpha(x) &\ge (- t_+)  \nabla_u \gamma_1(x) + C_+ \\
\limsup_{t\uparrow 0} D_{t,u}\alpha(x) &\le (- t_+)  \nabla_u \gamma_1(x) + C_+. 
\end{align*}
Combining this with \eqref{af}, we get the differentiability of $\alpha$ at $x$ in the direction $u$.\\
Next, choose any unit vector $v\in V$ orthogonal to $u$ and let $\gamma_2(x) := \gamma(x) \cdot v$. We want to show that $\nabla_u \gamma_2(x)$ exists. We proceed just as before; for some $k \in \N$, there exists $y_1,...,y_k \in \Gamma_x \setminus \{x\}$, $p_1,...,p_k \geq 0$, $q_1,...,q_k \geq 0$, $\Sigma p_i = 1$, $\Sigma q_i = 1$, $t_+ > 0$, $t_- <0$, such that 
\begin{align*}
x + t_+v = \Sigma p_i y_i\\
x + t_-v = \Sigma q_i y_i.
\end{align*}
By summing up the $y_i$'s and the weights $p_i$'s or $q_i$'s as before, we get this time  
\begin{align}\label{dual4a} 
D_{t,u}\alpha(x) \ge t D_{t,u}\gamma_1(x) - t_\pm D_{t,u}\gamma_2(x) + C_\pm(t)  \,\,\text{if}\,\, t>0\\
\label{dual4b} D_{t,u}\alpha(x) \le  t D_{t,u}\gamma_1(x) - t_\pm D_{t,u}\gamma_2(x) + C_\pm(t)   \,\,\text{if}\,\, t<0
\end{align}
 Taking $\limsup_{t\downarrow 0}$ in \eqref{dual4a} and $\liminf_{t\uparrow 0}$ in \eqref{dual4b}
 and recalling $t_+>0, t_-<0$,  and the existence of $\lim_{t\to 0} D_{t, u} \gamma_1 (x)$, we see that 
\begin{align*}
\nabla_u\alpha(x) &\ge (- t_+)  \liminf_{t\downarrow 0} D_{t,u}\gamma_2(x) + C_+ \\
\nabla_u\alpha(x) &\ge (- t_-)  \limsup_{t\downarrow 0} D_{t,u}\gamma_2(x) + C_- \\
\nabla_u\alpha(x) &\le (- t_+)  \limsup_{t\uparrow 0} D_{t,u}\gamma_2(x) + C_+ \\
\nabla_u\alpha(x) &\le (- t_-)  \liminf_{t\uparrow 0} D_{t,u}\gamma_2(x) + C_-, 
\end{align*}
which implies differentiability of $\gamma_2 = \gamma \cdot v$ at $x$ in the direction $u$. Now choose an orthonormal basis $\{u, v_1,...,v_m\}$ of $V$ and write proj$_V \gamma = (\gamma \cdot u)u + \Sigma (\gamma \cdot v_i)v_i$. We observed that each component of proj$_V \gamma$ is directionally-differentiable. This completes the proof. 
\end{proof}

\begin{remark}

For the maximization problem, we need to reverse the inequalities in \eqref{lower} and \eqref{dual11}, and then  proceed in the same way. Hence, the lemma is proved.
\end{remark}

 We now restrict our attention to the cases $c(x,y) = \pm|x - y|$ in trying to describe the profile of a set $\Gamma$ that is a $c$-contact layer.
 
 \begin{lemma} 
\label{lem: reg to Choquet}
Let $\Gamma \in \mathcal{S}_{MT}$, $\Omega$ an open set in $\R^d$ containing $X_\Gamma$, $\alpha : \Omega \rightarrow \R$ and $\gamma : \Omega \rightarrow \R^d$ be two functions. 
 Let $\beta : \R^d \rightarrow \R$ be either   
\begin{align}\label{ab}
\beta (y) = \sup_{x\in \Omega} \{|x - y| + \gamma(x) \cdot (y-x) + \alpha(x)\};
\end{align}
or 
\begin{align}\label{ac}
\beta (y) = \inf_{x\in \Omega} \{|x - y| + \gamma(x) \cdot (y-x) + \alpha(x)\}.
\end{align}
Assume that $\Gamma$ satisfies
\begin{align}\label{aa}
\beta (y) = |x - y| + \gamma(x) \cdot (y-x) + \alpha(x) \,\text{ for all } \,(x,y) \in \Gamma .
\end{align}
If $\alpha$ and $\gamma$ are differentiable at $x\in X_\Gamma$, then   the closure $\overline{\Gamma_x}$ coincides with the  set of extreme points of the convex hull of $\overline{\Gamma_x}$, i.e., 
$\overline{\Gamma_x} = {\rm Ext} \, \big{(}{\rm conv}(\overline{\Gamma_x})\big{)}$. 
\end{lemma}
\begin{proof}
First note that, for any closed set $A$ in $\R^d$, it is clear that ${\rm Ext} \, \big{(}{\rm conv}(A)\big{)} \subset A$. To show the 
 reverse inclusion in our setting, we define the ``tilted cone"
\begin{align*}
\zeta (x,y) = \zeta_x (y) = \zeta_y (x) := |x - y| + \gamma(x) \cdot (y-x) + \alpha(x).
\end{align*}
The duality condition  \eqref{aa} with \eqref{ab} tells us the following: if $(x,y) \in \Gamma$, then for all $x' \in \Omega$,
\begin{align} \label{ad}
\zeta_{x'} (y) \leq \zeta_x (y).
\end{align}
Or \eqref{aa} with \eqref{ac} we get the reverse inequality.\\
Note that since $\zeta_x (y)$ is continuous, the same inequality holds for all $y \in \overline{\Gamma_x}$.
This obviously implies that,  if  $y \in \overline{\Gamma_x}$ and $x \neq y$, then the gradient with respect to $x$ vanishes:
\begin{align}
\nabla \zeta_y (x) = 0
\end{align}
and in fact \eqref{ad} also implies that  if  $y \in \overline{\Gamma_x}$, then necessarily $x \neq y$. (If $x=y$, then the function 
$\zeta_y (x)$ strictly increases as $x$ moves along the direction 
$\nabla_x [\gamma(x) \cdot (y-x) + \alpha(x)]$.)
We may call this as non-staying property or unstability,  for the maximization problem. For the minimization problem,  without loss of generality we already assumed that $x \notin \Gamma_x$, but in fact $x \notin \overline{\Gamma_x}$ as well, by \eqref{ae} below.\\
Now suppose the lemma is false. Then we can find $\{y, y_0,..., y_s\} \subset \overline{\Gamma_x}$ for some $s\geq1$ with $y=\Sigma^s _{i=0} p_i y_i$, $\Sigma^s _{i=0} p_i=1$. Choose a minimum $s$ such that all $p_i > 0$. Now taking directional derivative in the direction $u = \frac{x-y}{|x-y|}$ gives
\begin{align*}
\nabla_u \zeta_y (x) = \nabla_u \zeta_{y_i} (x) = 0 \,\,\forall i=0,1,...,s. 
\end{align*}
We compute
\begin{align*}
\nabla_u \zeta_{y_i} (x) = \frac{x-y_i}{|x-y_i|} \cdot u + \nabla_u \gamma(x) \cdot (y_i-x) - \gamma(x)\cdot u + \nabla_u \alpha (x).
\end{align*}
Then, by the linearity of $y \mapsto \nabla_u \gamma(x) \cdot y$, the equation $\nabla \zeta_y (x) = 0$ simply becomes
\begin{align*}
1 = \sum^s _{i=0} p_i \frac{x-y_i}{|x-y_i|} \cdot \frac{x-y}{|x-y|}.
\end{align*}
As $\frac{x-y}{|x-y|}$ is a unit vector and all $p_i > 0$, this can hold only if all $y_i$ lie on the ray emanated from $x$. The minimality of $s$ then implies that $s=1$, hence $\{y, y_0,y_1\} \subset \overline{\Gamma_x}$ would lie on a ray emanating from $x$, which is a contradiction, once we prove the following claim:
\begin{align}\label{ae}
\text{$\overline{\Gamma_x}$ is contained in the topological boundary of the closed convex hull of $\overline{\Gamma_x}$.}
\end{align}
Recall that here the topology is not the topology in $\R^d$ but the topology in $V:=V(\Gamma_x)$.  If  our claim is false and assuming first that dim$(V) \geq 2$, we can find  $y \in \Gamma_x \cap IC(\Gamma_x)$ as a barycenter of  a triangle joining 3 points $y_0,y_1,y_2$ in $\Gamma_x$. But the above argument implies that $y_0,y_1,y_2$ have to be aligned, which is a contradiction.
 If dim$(V) = 1$, then  as $x \in IC(\Gamma_x)$, we can find $\{y, y_0,y_1\} \subset \Gamma_x$ such that $x$ and $y$ are in the interior of the line segment $\overline{y_0 y_1}$. But then again by above, $\{y, y_0,y_1\}$ must lie on the ray (i.e. half-line) emanated from $x$, a contradiction. Finally, we cannot have dim$(V) = 0$ since this simply means that $\Gamma_x = \{x\}$, 
  but as we already showed above that $x \notin \Gamma_x$ in the case of maximization, while we already assumed without loss of generality that $x \notin \Gamma_x$ in the case of minimization. 
\end{proof}

Finally,  
the following result follows immediately from Theorem~\ref{thm: alpha beta difference}, Lemmata~\ref{lem: alpha gamma} and 
~\ref{lem: reg to Choquet}.

\begin{corollary}\label{structure} Let $c(x,y)=\pm |x-y|$ and assume $\Gamma$ is a $c$-contact layer in ${\mathcal S}_{MT}$. 
 If   $X_\Gamma \subset \Omega :=  IC(Y_\Gamma)$ with $\Omega$ being an open set in $\R^d$, then 
for $\mathcal{L}^d$ - a.e. $x$ in $\Omega$,   the closure $\overline{\Gamma_x}$ coincides with the  set of extreme points of the convex hull of $\overline{\Gamma_x}$, i.e., 
$\overline{\Gamma_x} = {\rm Ext} \, \big{(}{\rm conv}(\overline{\Gamma_x})\big{)}$. 
 \end{corollary}  
\section{Structure of optimal martingale supporting sets when the dual is attained}\label{ss: structure from dual}

The goal of this section is to prove Theorem~\ref{thm: main extremal}  
which shows that dual attainment in the optimization problem \eqref{MGTP} implies that  any optimal martingale transport is concentrated on a $c$-contact layer, and therefore has a specific extremal structure. We start by collecting the properties verified by a well chosen concentration set of a martingale measure. The proof is given in Appendix (A).

\begin{lemma}\label{treatment}
Let $\pi \in$ MT$(\mu, \nu)$ and let $\Lambda \subset \R^d\times \R^d$ be a Borel set with $\pi(\Lambda) =1$. Then there exists a Borel set $\Gamma \subset \Lambda$ with $\pi(\Gamma) =1$ such that the map $x\mapsto \pi_x$ is measurable and defined everywhere on $X_\Gamma$ in such a way that:
\begin{enumerate}
\item $\overline{\Gamma_x} = \supp \pi_x$  for all $x \in X_\Gamma$,
\item $\Gamma \in {\mathcal S}_{MT}$, that is $x \in IC(\Gamma_x)$ for all $x \in X_\Gamma$,
\item If we assume that $\mu << \mathcal{L}^d$, then $\Gamma$ can be chosen in such a way that $X_\Gamma \subset IC(Y_\Gamma)$.
\item If in addition $\pi$ is a solution of the optimization problem \eqref{MGTP}, then $\Gamma$ can be chosen to be finitely $c$-exposable.
\end{enumerate}
\end{lemma}

This leads us to use  the following terminology.

\begin{definition}\label{def: mg mon}\rm 
Let $\pi $ be a martingale transport plan in MT$(\mu, \nu)$. We shall say that 
\begin{enumerate}
\item $\Gamma$ is a {\it regular concentration set for $\pi$} 
if $\Gamma$ satisfies (1), (2), (3) in Lemma \ref{treatment}. 
\item 
 $\Gamma$ is a {\it martingale-monotone regular concentration set for $\pi$} (or simply $\Gamma$ is {\it martingale-monotone regular for $\pi$}) if $\Gamma$ also satisfies (4).
\end{enumerate}
\end{definition}

As mentioned in the introduction, there is a dual formulation for problem (\ref{opt}), just like in the Monge-Kantorovich theory for (non-martingale) mass transport.  

 \begin{lemma} 
 (see e.g. \cite{BeHePe11})\label{lem: duality gap} 
 Let $\mu$ and $\nu$ be two probability measures on $\R^d$ in convex order, and let $c:\R^d\times \R^d \to \R$ be a cost function that is lower semi-continuous, then 
\begin{eqnarray}\label{No.gap}
&&\min\left\{ \int_{\R^d\times \R^d} c(x,y) \,d\pi;\, \pi\in MT(\mu,\nu)\right\} \\
&& \qquad \qquad   =\sup \left\{\int_{\R^d}\beta d\nu -\int_{\R^d}\alpha d\mu;\,\, (\alpha, \gamma, \beta) \in E_m\, \,  \hbox{for some $\gamma \in C_b(\R^d, \R^d)$} \right\},\nonumber
\end{eqnarray}
and the minimization problem is attained at some martingale transport $\pi$. A similar result holds for the cost maximization problem, provided $c$ is upper semi-continuous, and $E_m$ is replaced by $E_M$. Furthermore, 
 
 \begin{enumerate} 
 
 \item If the dual problem is attained, then there is a concentration set $\Gamma$ for $\pi$ that is a $c$-contact layer. 
 
 \item Conversely, if $G\subset \R^d\times \R^d$ is a $c$-contact layer and $\pi^*(G)=1$ for some  $\pi^*\in$ MT$(\mu,\nu)$, then $\pi^*$ is an optimal martingale transport.
 
 \end{enumerate} 
 \end{lemma} 
 \begin{proof}  For \eqref{No.gap} see \cite{BeHePe11}. Let us show the items (1) and (2).  Note that if the dual problem is attained at functions $\alpha$, $\beta$ such that the triplet $(\alpha, \gamma, \beta)$ is in $E_m(c, \R^d, \R^d)$, then since 
\begin{equation}\label{basic.ineq}
 \beta(y) -\alpha(x) -\gamma(x)(y-x) \leq c(x,y) \quad  \hbox{for all $(x, y)\in \R^d\times \R^d$},
 \end{equation}
$\mu$ and $\nu$ are the marginals of some optimal $\pi$ in  MT$(\mu,\nu)$, and $\int \gamma(x) \cdot (y-x)  \,d\pi(x,y) = 0$ (due to the martingale condition),   we then have 
$$ \int_{\R^d\times \R^d}\{\beta(y)-\alpha(x)  -\gamma(x)(y-x)\}d\pi(x,y) =\int_{\R^d\times \R^d} c(x,y)d\pi (x,y).$$
It follows that 
\begin{equation}\label{basic.eq}
 \beta(y) -\alpha(x) -\gamma(x)(y-x) = c(x,y) \quad \hbox{for $\pi$ a.e \, $(x, y)\in \R^d\times \R^d$},
 \end{equation}
hence the equality holds on  a concentration set $\Gamma$ of $\pi$. 

Conversely, if $G\subset \R^d\times \R^d$ and $\pi^*(G)=1$ for some  $\pi^*\in$ MT$(\mu,\nu)$ and if there exists a triplet $(\alpha, \gamma, \beta)$ in $E_m(c, X_G, Y_G)$ with equality  (\ref{basic.eq}) holding on $G$, then $\pi^*$ is an optimal solution of the primal problem in 
(\ref{No.gap}).  Indeed, let $\pi \in$ MT$(\mu,\nu)$ and let $H$ be   such that $\pi(H)=1$. As $\mu(X_G) = 1$ and 
$\nu(Y_G) = 1$, by restriction we can then assume that $X_H \subset X_G$ and  $Y_H \subset Y_G$, hence by  integrating 
(\ref{basic.ineq}) with $\pi$, we get
\begin{align*}
\int_{\R^d\times \R^d} c(x,y) \,d\pi(x,y) \   \ge 
\int_{\R^d} \beta(y) \,d\nu(y) - \int_{\R^d}\alpha(x) \,d\mu(x).
\end{align*}
(Again $\int \gamma(x) \cdot (y-x)  \,d\pi(x,y) = 0 $ since $\pi \in$ MT$(\mu,\nu)$).  However, by integrating (\ref{basic.ineq}) with $\pi^*$ and since we have equality on $G$, we get
\begin{align*}
\int_{\R^d\times \R^d} c(x,y) \,d\pi^* = \int_{\R^d\times \R^d}\{\beta(y)-\alpha(x)  -\gamma(x)(y-x)\}d\pi^*= \int_{\R^d} \beta(y) \,d\nu - \int_{\R^d}\alpha(x) \,d\mu.
\end{align*}
This shows that $\pi^*$ is optimal.  Hence, every martingale measure that is concentrated on a $c$-contact layer is optimal. On the other hand, there exist optimal martingale measures that do not concentrate on  $c$-contact layers (\cite{BeHePe11}).  
\end{proof}
This suggests that dual attainability is actually a property of the support of the optimal martingale transport and not of the measure itself.
Now an obvious but important remark is that any subset of a $c$-contact layer is also a $c$-contact layer. The same holds for dual attainment in the martingale transport problem. 
Indeed,  if $\pi \in MT(\mu, \nu)$  and $B$ is a Borel set, we denote by $\pi_B$ its restriction on $B \times \R^d$, and we let $\mu_B, \nu_B$ be the first and second marginals of $\pi_B$.  Then we introduce the following: 

\begin{definition}\label{def: local problem}\rm 
Let $\pi \in MT(\mu, \nu)$ be given, and let $B$ be a Borel set.  We say that an admissible triple  $(\alpha, \gamma, \beta) \in E_m(c, B, \R^d)$  is  {\em $c$-dual to $\pi$ on $B$}, if the following holds:
 \begin{align}\label{eq: dualineq12} 
\int_{\R^d} \beta(y) \,d\nu_B (y)  - \int_{\R^d}\alpha (x)\,d\mu_B (x)   = \int_{\R^d\times \R^d} c(x,y)\,d\pi_B (x,y).
\end{align}
If such a triple exists, then we say that $\pi$ admits a {\em c-dual on $B$}. Note that in this case, 
 $\pi_B(\Gamma_{(\alpha, \gamma, \beta)})=\mu(B) $, 
 that is,  $\pi_B$ is concentrated on a $c$-contact layer.
\end{definition}
Now, we can deduce the following. 

\begin{theorem} \label{relaxed} Let $c(x,y)=\pm |x-y|$ and $\mu$ be a probability measure that is absolutely continuous with respect to the Lebesgue measure. If $\pi\in MT(\mu, \nu)$ is a solution of \eqref{MGTP} for either the minimization or maximization problem that admits a $c$-dual on a Borel subset $B$, then for $\mu$-almost all $x \in B$,  ${\rm supp}\, \pi_x = {\rm Ext}\,\big{(} {\rm conv}({\rm supp}\, \pi_x) \big{)}$
.
\end{theorem} 
\begin{proof}  Let $(\alpha, \gamma, \beta)$ be a $c$-dual to $\pi$ on $B$ and let $\Lambda$ be its contact layer. Then $\Lambda$ contains the full measure (that is, $\mu(B)$) of $\pi_B$. Apply Lemma  \ref{treatment} to get $\Gamma \subset \Lambda$ in such a way that $\pi_B (\Gamma) = \mu(B)$, $\Gamma \in {\mathcal S}_{MT}$,  $X_\Gamma \subset \Omega := IC(Y_\Gamma)$ and $\supp \pi_x = \overline{\Gamma_x}$ for $\mu$ a.e. $x \in B$. Now since $\Gamma$ is also a $c$-contact layer, Corollary \ref{structure} applies to get the claimed result.
 \end{proof}

\begin{remark} \label{local} Note that the above theorem shows that Conjecture (1) is valid provided {\it duality is attained locally}. In other words, if for any $x$ in the support of $\mu$, there exists a ball $B$ centered at $x$ such that the optimal martingale measure $\pi$ admits a $c$-dual on $B$. This refinement will be used in the next section. On the other hand, there exists an optimal martingale measure where ``local dual attainment" does not hold on any neighborhood. This can be seen with the following example given in  \cite{BeHePe11}. 
\end{remark} 

\begin{example}\label{ah}
 Let $\mu= \nu$ be two identical probability measures on the interval $[0,1]$, then the only martingale (say $\pi$) from $\mu$ to itself is the identity transport, hence it is obviously the solution of the maximization problem with respect to the distance cost, and its support is $\Gamma = \{(x,x) : x \in [0,1]\}$. If now $\{\alpha, \gamma, \beta\}$ is a solution to the dual problem, then 
\begin{align*}
\beta(y)&\geq |x-y|+ \gamma(x) \cdot (y-x) + \alpha(x)\,\, \forall x \in [0,1], \forall y \in [0,1] ; \\
\beta(y)&= |x-y| + \gamma(x) \cdot (y-x) + \alpha(x)\,\, \forall (x,y) \in \Gamma . 
\end{align*}
The above relations easily yield that for any $0<a<b<1$, we have 
$\gamma(a) +2 \leq \gamma(b)$, which means that 
it is impossible to define a suitable real-valued function $\gamma$ for a.e. $x$ in $[0,1]$.
\end{example}

\section{When the marginals are in subharmonic order}\label{S: local dual}

In this section, we consider a case where the dual martingale problem is attained --at least locally-- which will allow us to apply Theorem \ref{relaxed} and verify that conjecture (1) holds in that particular case. We consider the following ``balayage order" between probability measures, that is stronger and more natural than the convex order, at least in higher dimensions. 
We say that probability measures $\mu$ and $\nu$ are in {\it subharmonic order}, 
$\mu \le_{SH} \nu$,  if
\begin{equation}\label{PSH}
\hbox{$\int_{\R^d} \varphi\,d\mu\leq \int_{\R^d} \varphi\, d\nu$ for every subharmonic function $\varphi$  on $\R^d$.} 
\end{equation}
For simplicity, we shall assume that $\mu$ and $\nu$ have compact support so as to avoid integrability issues. 
Since convex functions are subharmonic, it is clear that  $\mu \leq_{SH} \nu \Rightarrow \mu \leq_{C} \nu$ and that the two notions are equivalent in one-dimension.\\
  Note that if $(B_t)_t$ is a $d$-dimensional Brownian motion with initial distribution $\mu$ and if $\nu$ is the distribution of $B_T$ where $T$ is a stopping time such that $(B_{T\wedge t})_t$ is a uniformly integrable martinagle, then $\mu \leq_{SH} \nu$. Such stopping times are normally called {\it standard}. The converse is also true  and belongs to a family of results known as Skorokhod embeddings (e.g., see Ob{\l}{\'o}j \cite{Obloj}). 
   In other words, (\ref{PSH}) is essentially equivalent to 
  \begin{equation}
  \hbox{$\mu \sim B_0$ and $\nu \sim B_T$ for a (possibly randomized) standard stopping time $T$.}
  \end{equation} 
We now consider the  
Newtonian potential (or simply, potential) $P_\mu$ of a probability measure $\mu$ with compact support, that is 
$$ P_\mu (x) = \frac{1}{d(2-d)w_d}\int_{\R^d} |x-y|^{2-d} \,d\mu(y), $$ 
in such a way that in dimension $d\geq 3$, we have $\Delta P_\mu =\mu$ (in the sense of distributions). Note that (\ref{PSH}) then implies 
that 
\begin{equation}
P_\mu(x) \leq P_\nu(x), \quad \forall x \in \R^d.
\end{equation}
The converse is also true at least for $d\geq 3$. See Falkner \cite{Falk}.  \\
Finally, note that if we consider an elliptic operator
$ L_t = \sum_{ij} a_{ij}(t) \partial_i \partial_j$
corresponding to a one-parameter family of positive matrices $(a_{ij}(t))$, $t >0$, and if $\mu$, $\mu_t$ are measures with densities $\rho$, $\rho_t$ respectively, where
 \begin{align}\label{eq: parabolic}
 \begin{cases}
\partial_t \rho_t - L_t \rho_t = 0      & \text{ for $t>0$ and  in $\R^d$}, \\
    \rho_0 = \rho,  & \text{}
\end{cases}
\end{align}
then one can easily verify that 
 $\mu\le_{SH} \mu_t$. Actually, one can show that 
  \begin{equation}
P_\mu(x)<P_{\mu_t} (x),  \quad \forall x \in \R^d. 
\end{equation}
The importance of such a strict inequality will be clear thereafter. 
The following is the main result of this section.
\begin{theorem}\label{cor: local Choquet in SH}
Assume $\mu\le_{SH} \nu$ where $\mu, \nu$ are probability measures with compact support on $\R^d$ such that  $\mu <<\mathcal{L}^d$. Assume the function $P_\nu - P_\mu$ lower semi-continuous and consider the open set $U: =\{ x \in \R^d \ | P_\nu(x) - P_\mu(x) >0\}$. If $\pi \in M(\mu, \nu)$ is an optimal solution for the minimization problem  \eqref{MGTP}, where  the cost function is either  $c(x,y) = |x-y|$ or $c(x,y) =- |x-y|$,  then:
\begin{enumerate} 
 
 \item For each $x \in U$, there exists a ball $B$ centered at $x$ such that $\pi$ admits a $c$-dual on $B$.  
 \item For $\mu$ - a.e. $x \in U$,  ${\rm supp}\, \pi_x = {\rm Ext}\,\big{(} {\rm conv}({\rm supp}\, \pi_x) \big{)}$.
  In particular, Conjecture (1) holds if $\mu(U)=1$. 
 \end{enumerate} 

\end{theorem}
\begin{remark}\label{rmk: diffusion}
The assumption that $\nu$ is compactly supported can be replaced with appropriate decay conditions on $P_\nu- P_\mu$ and $\nabla (P_\nu - P_\mu)$. In particular, Conjecture 1 holds for $\mu$ and $\nu = \mu_t$ from the diffusion example in \eqref{eq: parabolic} for $d\ge 3$, if the initial measure $\mu$ is absolutely continuous and compactly supported. Note that the 2-dimensonal case is true in full generality, that is when the marginals are simply in convex order (See Section 7).
\end{remark} 

\begin{proof}[Proof of Thereom~\ref{cor: local Choquet in SH}] Denoting $E_m=E_m(c, \R^d, \R^d)$, we have from Lemma~\ref{lem: duality gap} that  
 \begin{eqnarray}\label{No.gap.1}
&&l := \min\left\{ \int_{\R^d\times \R^d} c(x,y) \,d\pi;\, \pi\in MT(\mu,\nu)\right\} \\
&& \qquad \qquad   =\sup \left\{\int_{\R^d}\beta d\nu -\int_{\R^d}\alpha d\mu;\,\, (\alpha, \gamma, \beta) \in E_m\, \,  \hbox{for some $\gamma \in C_b(\R^d, \R^d)$} \right\}.
\end{eqnarray}
Let $\pi$ be an optimal solution for the minimization problem, and let $\Gamma$ be a martingale-monotone regular concentration set for $\pi$ (as in Defintion~\ref{def: mg mon}).
Fix a bounded open set $\Omega$ which is sufficiently large such that $\supp (\mu) \subset \Omega$.
We shall show that  for each $x_0 \in U\cap \Omega$, there exists a ball $B=B(x_0) \subset U\cap \Omega$ centred at $x$, such that $\pi$ has a $c$-dual on $B$. \\
For that consider a maximizing sequence for the dual problem, that is admissible triples  $(\alpha_n,  \gamma_n,  \beta_n)\in E_m(c, \R^d, \R^d)$ such that 
\begin{align}\label{eq: l}
 l = \lim_{n\to \infty}  \int \beta_n d\nu  - \int \alpha_n d\mu . 
\end{align}
In view of Theorem~\ref{thm: alpha beta difference} and remark \ref{localization}, we can assume that the triplet $(\alpha_n, \gamma_n, \beta_n)\in E_m(c, \Omega, \R^d)$, that  
$\alpha$ is Lipschitz  in $\Omega$, $\gamma$ is bounded in $\Omega$, and that
 \begin{equation}\label{same}
 \beta (x)  \le   \alpha (x)  \le \beta(x) +  \delta \quad \hbox{for all $x \in \Omega$, }
 \end{equation}
  where $0\leq \delta \leq 2 {\rm diam} (\Omega)$. Note that $\delta=0$ if $c(x,y) = |x-y|$. \\
 We consider the convex function
\begin{align*}
\chi_n (y) : = \sup_{x \in \Omega} \,\{ -\alpha_n(x)  - \gamma_n (x) \cdot (y-x)\}.
\end{align*}
Since  $\alpha_n$ and $\gamma_n$ are bounded and the set $\Omega$ is bounded, the functions
$\chi_n$ are Lipschitz on $\R^d$. 
Note also that by adding a sequence of affine functions $L_n$ (since $L_n(y) = L_n(x) + \nabla L_n(x) \cdot (y-x)$) the new sequence $(\alpha_{n} + L_{n}, \beta_{n}+L_{n}, \gamma_{n} + \nabla L_{n})$ will still have the same properties. By adding appropriate affine function $L_n$ and their gradients to the triple $(\alpha_n,  \gamma_n, \beta_n)$, we may therefore assume  that  
$$\chi_n (x) =0 \quad \hbox{and \quad $\chi_n \ge 0$ for every $n$. }
$$
We now show that a subsequence of $\alpha_n, \gamma_n$ converge locally in $U\cap \Omega$.
We first establish suitable estimates on $\chi_n$. Consider the Lipschitz function $q(y):= \sup_{x \in \Omega} c(x,y)$ and note that 
\begin{align}\label{eq: chi ineq}
- \alpha_n (y) \le \chi_n (y) \quad \forall y \in \Omega \quad \text{ and }\quad   \chi_n (y)       \le  q(y) - \beta_n (y) \quad \forall y \in \R^d.
\end{align}
Hence,
\begin{align*}
0\le  \int \chi_n (d\nu - d\mu) \le  -\int \beta_n d \nu + \int \alpha_n d\mu +  C_1, 
\end{align*}
where $C_1=\int q(y)d\nu(y) < \infty$,  since $q$ is Lipschtiz and $\nu$ has finite first moment.
Since $(\alpha_n,  \gamma_n, \beta_n)$ is a maximizing sequence, then 
for all sufficiently large $n$,
\begin{align*}
 0\le  \int \chi_n  (d\nu - d\mu ) \le  -l+ C_1 +1 =: C_2.
\end{align*}
Hence,
\begin{align}\label{eq: Delta upper}
C_2 \ge  \int \chi_n ( d\nu - d\mu)  = \int \chi_n \, \Delta (P_\nu-P_\mu) =
 \int \Delta \chi_n  (P_\nu - P_\mu)
 \end{align}
where $\Delta \chi_n$ is the distributional Laplacian of the convex function $\chi_n$. For the second last equality 
note that $\Delta P_\mu =\mu, \, \Delta P_\nu =\nu$, and for the last equality note that $\chi_n$ is convex Lipschitz and $ P_\nu-P_\mu$, $\nabla (P_\nu- P_\mu)$ decays to zero at infinity  by assumption, enabling us to integrate by parts. \\
Now fix  $x_0 \in U\cap \Omega$ and pick  a closed ball $B:=B_r (x_0) \subset U \cap \Omega$ of radius $r$, centered at $x_0$.  
Since $P_\nu-P_\mu$ is lower-semicontinuous and strictly positive on $U$, we have $\epsilon_{B} : = \min_{B} [P_\nu -P_\mu] >0$  
which, in view of \eqref{eq: Delta upper}, implies  that  
\begin{align*}
 \int_{B_r(x_0)} \Delta \chi_n  \le\frac{C_2}{ \epsilon_{B}}. 
\end{align*}
Now, modulo approximating it by smooth convex function, we can assume that $\chi_n$ is smooth and apply  Proposition~\ref{prop: Lap bound}
 to conclude that $\chi_n$ is bounded in a smaller ball $B_{r'}(x_0)$, uniformly in $n$. 
In view of (\ref{same}) and \eqref{eq: chi ineq}, the uniform boundedness of $\chi_n$ then implies the uniform boundedness of $\alpha_n, \beta_n$ on $B_{r'}(x_0)$. 
 Moreover, since 
\begin{align*}
 -\alpha_n(x)  - \gamma_n (x) \cdot (y-x) \leq \chi_n (y) \le C, \quad \forall x,y \in B_{r'}(x_0),
\end{align*}
  we also find that $\gamma_n$ is  uniformly bounded in $n$ on a smaller ball $B=B_{r''}(x_0)$, $r'' < r' < r$, in such a way that the sequences $(\alpha_n, \gamma_n,  \beta_n)_n$ are all uniformly bounded on $B$. \\
Apply now Koml\'os theorem, which states that every $L^1$-bounded sequence of real functions has a subsequence such that the arithmetic means of all its subsequences converge pointwise almost everywhere. Since the arithmetic means of  $\alpha_n$, $\beta_n$, $\gamma_n$ also yield a maximizing sequence of admissible triples for \eqref{eq: l}, 
we can therefore assume that the original functions $\alpha_n$, $\beta_n$ and $\gamma_n$ converge $\mathcal{L}^d$ a.e. in $B$ to, say, $\overline{\alpha}$, $\overline{\beta}$, and $\overline{\gamma}$ on $X \subset B$ where $\mathcal{L}^d (B \setminus X) =0$. Notice that these limits are bounded in $X$. \\
It is not clear, however, that this triple $(\overline{\alpha}, \overline{\gamma}, \overline{\beta})$ will give the desired one, especially because $\overline{\beta}$ is only defined in $X$, not in $\R^d$.   We thus proceed as follows.
Define
\begin{align*}
 \beta_{X, n} = \inf_{x \in X} \{ c(x, y) + \alpha_n(x) + \gamma_n (x) \cdot (y-x)\}.
\end{align*}
Notice that since  $\alpha_n$, $\gamma_n$ are bounded in $X$ uniformly in $n$ and $y\mapsto c(x, y)$ is Lipschitz in $\R^d$ with uniformly bounded  Lipschitz constants for $x \in X$, we immediately see that 
the function  $y \in \R^d \mapsto \beta_{X,n}(y)$ is Lipschitz (uniformly in $n$) and is uniformly bounded on each compact set. Therefore, there exists a subsequence, which we still denote by $\beta_{X,n}$, that converges to a Lipschitz function $\overline{\beta}_X$ uniformly on each  
compact set in $\R^d$. 
Moreover, from the definition of $\beta_{X,n}$, the triple $(\alpha_n, \gamma_n,  \beta_{X,n})$ satisfy 
\begin{align*}
  \beta_{X,n}(y)-\alpha_n(x)  -\gamma_n(x)(y-x) \le c(x,y) \quad \forall (x, y) \in X \times \R^d.
\end{align*}
Thus $(\alpha_n, \gamma_n,  \beta_{X,n}) \in E_m(c, X, \R^d)$. Also, taking the limit as $n\to \infty$, the above inequality still holds in the limit, and so  the triple
$(\overline{\alpha}, \overline{\gamma}, \overline{\beta}_X)\in E_m ( c, X, \R^d)$. \\ 
To show that  the triple  $(\overline{\alpha}, \overline{\gamma}, \overline{\beta}_X)$  is a $c$-dual to $\pi$ on $B$ (in the sense of Definition~\ref{def: local problem}), it remains to verify 
\eqref{eq: dualineq12}. For this, observe from the definition of $\beta_{X,n}$ that 
$\beta_n (y) \le  \beta_{X, n} (y)$ for all $ y \in \R^d.$
Thus, 
\begin{align*}
 \int \beta_n d\nu_B -\int \alpha_n d\mu_B \le \int \beta_{X,n} d\nu_B - \int \alpha_n d\mu_B  \le \int c(x, y) d\pi_B(x,y).
\end{align*}
Noting that the maximizing sequence of admissible triple $(\alpha_n,  \gamma_n,  {\beta}_n)$  for 
\eqref{eq: l} is also a
maximizing sequence of admissible triple for $\pi_B$, i.e.,
\begin{align*}
\lim_{n\to \infty}  \int\beta_{n} (y) \,d\nu_B(y) -\int\alpha_n(x)\, d\mu_B (x) &= \int c(x,y) \,d\pi_B (x,y),
\end{align*}
we therefore have that
\begin{align*}
\lim_{n\to \infty}  \int \beta_{X,n} (y) \,d\nu_B(y) -\int \alpha_n(x)\, d\mu_B (x) = \int c(x,y) \,d\pi_B (x,y).
\end{align*}
To bring the limit inside the integrals, recall that $\beta_{X,n}$ is uniformly Lipschitz (in $n$) and $\alpha_n$ is uniformly bounded, $\mu(B \setminus X)=0$ and $\nu_B$ has finite first moment. Thus, by the dominated convergence theorem, 
\begin{align*}
 \int \overline{\beta}_X (y) \,d\nu_B(y) -\int \overline{\alpha}(x)\, d\mu_B (x) = \int c(x,y) \,d\pi_B (x,y).
\end{align*}
Therefore, the triple $(\overline{\alpha}, \overline{\gamma}, \overline{\beta}_X)$ is a $c$-dual to $
\pi$ on $B$, proving the item (1). Then by Theorem \ref{relaxed}, for $\mu$ a.e. $x \in B$ we have that ${\rm supp}\, \pi_x = {\rm Ext}\,\big{(} {\rm conv}({\rm supp}\, \pi_x) \big{)}$. As $U$ can be covered by countably many such balls $B$,  for $\mu$ a.e. $x \in U$ we have that ${\rm supp}\, \pi_x = {\rm Ext}\,\big{(} {\rm conv}({\rm supp}\, \pi_x) \big{)}$, proving the item (2). 
\end{proof}

 \section{A canonical decomposition for the support of martingale transports}\label{s: decomposition}

We have shown in the last sections that Conjecture 1 holds whenever the dual problem is  (locally) attained. In this section, we shall decompose an optimal martingale transport $\pi$ into components on which an induced martingale transport problem is defined in such a way that its dual problem is attained. For that, we shall first associate to any Borel set $\Gamma \in \mathcal{S}_{MT}$ a unique irreducible convex paving $\Phi$. We then show that if every finite subset of $\Gamma$ is a $c$-contact layer (a property satisfied by a concentration set of an optimal martingale measure), then every  subset $\Gamma_C=\Gamma \cap (C \times \R^d)$ where $C$ is a component of the convex paving $\Phi$, is a $c$-contact layer. \\

\noindent{\bf 6.1 Irreducible convex pavings associated to martingale supporting sets:}
Let $\Gamma$ be a Borel set in ${\mathcal S}_{MT}$. We start by defining an equivalence relation on $X_\Gamma$. For each $x \in X:=X_\Gamma$, we define inductively an increasing sequence of convex open sets $(C_n(x))_n$ in the following way:

Start with the trivial equivalence relation $x\sim_0x'$ iff $x=x'$. 
Let $C_0(x) := IC(\Gamma_x)$ and recall that if $\Gamma_x = \{x\}$, then $C_0(x) = \{x\}$.
Now define the following equivalence relation on $X$:
$x \sim_1x' $ if there exist finitely many $x_1,...x_k$ in $X$ such that the following chain condition holds:
\begin{align*}
&C_0(x) \cap C_0(x_1) \neq \emptyset, \\
&C_0(x_i) \cap C_0(x_{i+1}) \neq \emptyset \,\,\forall i=1,2,...k-1,\\
&C_0(x_k) \cap C_0(x') \neq \emptyset . 
\end{align*}
We then consider the open convex hull:
$$C_1(x) := IC[\bigcup_{x' \sim_1 x} C_0(x')].$$
Note that $x \sim_1 x'$ implies $C_1(x) = C_1(x')$. Unfortunately, the convex sets $C_1(x)$ do not determine the equivalence classes. In particular, they may not be mutually disjoint for elements that are not equivalent for $\sim_1$. So,
we proceed to define $\sim_2$ in a similar way:
$x \sim_2 x' $ \, if there exist finitely many $x_1,...x_k$ in $X$ such that the following chain condition holds:
\begin{align*}
&C_1(x) \cap C_1(x_1) \neq \emptyset,  \\
&C_1(x_i) \cap C_1(x_{i+1}) \neq \emptyset \,\,\forall i=1,2,...k-1, \\
&C_1(x_k) \cap C_1(x') \neq \emptyset;
\end{align*}
and we set
$$C_2(x) := IC[\bigcup_{x' \sim_2 x} C_1(x')].$$
Again, $\sim_2$ is an equivalence relation and one can easily see that 
\begin{itemize}
\item $x \sim_1x' \Rightarrow x \sim_2x' $
\item $x \sim_2 x' \Rightarrow C_2(x) = C_2(x')$
\item $C_1(x) \subset C_2(x)$.
\end{itemize}
But still, the sets $C_2(x)$ may not be mutually disjoint for non-equivalent $x's$. We continue inductively in a similar fashion by defining equivalence relations $\sim_n$ for $n=1,2,...$ and their corresponding  classes 
$$C_n(x) := IC[\bigcup_{x' \sim_n x} C_{n-1}(x')].$$
It is easy to check that we have the following properties for each $n$, 
\begin{align*}
x \sim_n x' &\Rightarrow x \sim_{n+1} x' \\
x \sim_n x' &\Rightarrow C_n(x) = C_n(x') \\
C_n(x) &\subset C_{n+1}(x).
\end{align*}
Finally, define the  equivalence  relation  
$$\hbox{$x \sim x'$ if $x \sim_n x'$ for some $n$,}
$$
and its corresponding convex sets
\begin{align}\label{eq: C x}
 C(x) := \lim_{n \to \infty} C_n(x) = \cup_{n=0}^\infty C_n(x).
\end{align}
Now, we show that $\Psi=\{C(x)\}_{x\in X}$ is an irreducible convex paving for $\Gamma$. 
\begin{theorem}
\label{lem: partition} The canonical relation $\sim$ on $X_\Gamma$ and the components $(C(x))_ {x\in X_\Gamma}$ satisfy the following:
\begin{enumerate}
\item $x \sim x' \Rightarrow C(x) = C(x') $, and $x \nsim x' \Rightarrow C(x) \cap C(x') = \emptyset$.
\item $C(x)$ are mutually disjoint, that is either $C(x) = C(x')$ \,or  \,$C(x) \cap C(x') = \emptyset$.
\item $x' \in X \cap C(x)$ if and only if $x' \sim x$. 
\item $\Phi=\{C(x)\}_{x\in X}$ is an irreducible convex paving for $\Gamma$. 
\item $C_n(x) =  IC[\bigcup_{x' \sim_n x} \Gamma_{x'}]$ \,\,for $n\geq 0$\, and \,$C(x) = IC[\bigcup_{x' \sim x} \Gamma_{x'}]$.   In particular, $\Gamma_x \subset \overline{C(x)}$.
\end{enumerate}
\end{theorem}
\begin{proof}
The fact that $x \sim_n x' \Rightarrow C_n(x) = C_n(x')$ gives the first part of (1). If there exists a $z \in C(x) \cap C(x')$, then there is $N$ such that $z \in C_N(x) \cap C_N(x')$, implying $x \sim_{N+1} x'$ and verifying the second part of (1) of which (2) and (3) are obvious consequences.\\
To prove (4), let $\Psi$ be any convex paving of $\Gamma$ and let $z, x \in X_\Gamma$, $D \in \Psi$ be such that $C(z) \cap D(x) \neq \emptyset$. We must show that $C(z) \subset D(x) $. We claim that, for any $n \geq 0$, 
$$(\ast) \qquad \qquad C_n(z) \cap D(x) \neq \emptyset \Rightarrow C_n(z) \subset D(x), \text{\, for every }\,z,x \in X_\Gamma.$$
Indeed, it is true for $n=0$ by definition. Assume that $(\ast)$ is true for some $n$, and suppose 
$C_{n+1}(z) \cap D(x) \neq \emptyset$. Note that $C_{n}(z) \subset D(z)$, and so if $w \sim_{n+1} z$, by $(\ast)$ we have that $C_{n}(w) \subset D(z)$. As $C_{n+1}(z) = IC[\bigcup_{w \sim_{n+1} z} C_{n}(w)]$, this readily implies that 
$C_{n+1}(z) \subset D(z)$, but then $D(z) \cap D(x) \neq \emptyset$ and hence $D(z) = D(x)$. This proves $(\ast)$ for every $n \geq 0$. Now if $C(z) \cap D(x) \neq \emptyset$, then for all large $n$ $C_n(z) \cap D(x) \neq \emptyset$, hence by $(\ast)$ we get that $C_n(z) \subset D(x)$. Therefore $C(z) \subset D(x)$ which proves the irreducibility of $\Phi$.\\
For (5), let $(A_i)_{i \in I}$ be any family of sets in $\R^d$, where $I$ is an index set. Then it is easy to see that
$$IC(A_i) = IC(CC(A_i)), \text{ and}$$
$$CC(\bigcup_{i\in I} CC(A_i)) = CC(\bigcup_{i\in I} A_i) = CC(\bigcup_{i\in I} IC(A_i)).$$
But note that $A \subset B$ does not imply $IC(A) \subset IC(B)$ in general. The above implies in particular
$$IC(\bigcup_{i\in I}A_i) = IC(\bigcup_{i\in I} IC(A_i)).$$
In addition, a simple induction shows that for every $n \geq 0$, we have
$$C_n(x) = IC[\bigcup_{x' \sim_n x} \Gamma_{x'}]. $$
Indeed, it is true for $n=0$ by definition. Suppose $C_n(x) = IC[\bigcup_{x' \sim_n x} \Gamma_{x'}]. $ Now by definition,
$$C_{n+1}(x) = IC[\bigcup_{x' \sim_{n+1} x} C_n(x')] =  IC[\bigcup_{x' \sim_{n+1} x} IC( \bigcup_{x'' \sim_{n} x'} \Gamma_{x''})] =   IC[\bigcup_{x' \sim_{n+1} x} ( \bigcup_{x'' \sim_{n} x'} \Gamma_{x''})].   $$
But $\bigcup_{x' \sim_{n+1} x} \bigcup_{x'' \sim_{n} x'} \Gamma_{x''} = \bigcup_{x' \sim_{n+1} x} \Gamma_{x'}$, hence, $C_{n+1}(x) = IC(\bigcup_{x' \sim_{n+1} x} \Gamma_{x'})$, completing the induction.

Finally, we proceed as follows:
\begin{align*}
C(x) &= IC[C(x)] = IC[\bigcup_{n\geq 0} C_n(x)] = IC[\bigcup_{n\geq 0} IC(\bigcup_{x' \sim_n x} \Gamma_{x'})] \\
&= IC[\bigcup_{n\geq 0} \bigcup_{x' \sim_n x} \Gamma_{x'}] = IC[\bigcup_{x' \sim x} \Gamma_{x'}], 
\end{align*}
which completes the proof of (5) and the theorem.
\end{proof}

\noindent{\bf 6.2 When irreducible components are $c$-contact layers:}\label{ss: dual on each:}
Let $c : \R^d \times \R^d \to \R$ be a cost function on which we make no assumption.
Our aim is to prove Theorem~\ref{th:k},  which will follow from the following.

\begin{theorem}  \label{local.layers} Let $\Gamma \in \mathcal{S}_{MT}$ be $c$-finitely exposable.
If $\Phi$ is the irreducible convex paving of $\Gamma$, then for every convex component $C$ in $\Phi$, the set  $\Gamma \cap (C \times \R^d) = \Gamma \cap (C \times \overline{C})$ is a $c$-contact layer.
\end{theorem}
First, we prove the following lemma.

\begin{lemma}\label{3.14}  Let $\Gamma \in \mathcal{S}_{MT}$ be $c$-finitely exposable,
and denote $X:=X_\Gamma$. Fix $x_0\in X$ and 
set $G:=\Gamma \cap (C(x_0) \times \R^d)$, where $C(x_0)$ is the component of the irreducible convex paving $\Phi$ of $\Gamma$ that contains $x_0$.
Then, for each $y \in Y_G$, there exists a compact interval $K_y \subset \R$ such that any finite subset $H \subset G$ is a $c$-contact layer for a triplet $(\alpha, \gamma, \beta)$,   where $\beta(y) \in K_y$\, for all $y \in Y_H$.
\end{lemma}
The above lemma is essentially saying that there is some uniformity in the way $c$-admissible triplets can expose finite subsets of $G$ as $c$-contact layers. This control on the $\beta$ component of the $c$-admissible triplets will allow us to use Tychonoff's compactness theorem to deduce that the whole of $G$ is a $c$-contact layer.

To prove Lemma~\ref{3.14}, we first give an idea about the degrees of freedom we have in choosing $\beta$. 
First, note that if $\beta$ is $c$-admissible for $G$ (meaning that there is $\alpha, \gamma$ such that $G \subset \Gamma_{(\alpha, \gamma, \beta)}$) and $L: \R^d \to \R$ is an affine function, then 
$\beta - L$ is also a $c$-admissible for $G$. Letting $m = \dim(V(Y_G))$, we can find $\{y_0,...,y_m \} \subset Y_G$ such that 
$V(\{y_0,...,y_m \}) = V(Y_G)$, i.e. $\{y_0,...,y_m \}$ constitute vertices of an $m$-dimensional polytope in $V(Y_G)$.  Now for a given $c$-admissible  function $\beta$ for $G$, let $L : V(Y_G) \to \R$ be an affine function determined by $L(y_i) = \beta(y_i)$ for $i=0,1,...,m$. The function $\beta' := \beta - L$ then satisfies $\beta' (y_i) = 0$ for all $ i=0,1,...,m$, which means that we have $m+1$ freedom of choice on the value of $\beta$. In other words, if we set 
$K_{y_i }= \{0\}$ for $ i=0,1,...,m$, then we can find $\beta'$ such that $\beta'(y_i) \in K_{y_i }$ for each $y_i$. 
Now, we want to observe how the initial value of $\beta$ can control its values at other points $y$. We shall see that 
the control of the value of $\beta$ propagates well along a given chain inside the equivalent class $C(x_0)$.

 The proof of Lemma~\ref{3.14} is involved, and requires several key steps. To clarify the idea, we consider first the  special case $c=0$ where we can establish a complete control on the dual functions. 
\begin{lemma}\label{lem: trivial c}
Let $G \in \mathcal{S}_{MT}$ and assume that it is a $0$-contact layer for a triplet $(\alpha, \gamma, \beta)$, that is
\begin{align}
\label{sd1}\beta(y) &\geq L_x(y) \quad  \forall x\in X_G, y \in Y_G \\
\label{sd2}\beta(y) &= L_x(y) \quad  \forall (x,y) \in G,
\end{align}
where for each $x$, $L_x$ is the  affine function
\begin{align*}
 L_x(y) := \gamma(x) \cdot (y-x) + \alpha(x).
\end{align*}
Then, $L_x = L_{x'}$ on $V(C(x))$ whenever $x \sim x'$.  
\end{lemma}
Note that  \eqref{sd2} says  that if we have control on $L_x$, then we have control on $\beta$ for all $y\in G_x$.  In particular, Lemma~\ref{lem: trivial c} implies that if $L_x =0$ (we can choose such $L_x$ without loss of generality) then, $L_{x'}=0$ on $V(C(x))$ for all $x' \in C(x)$, thus $\alpha(x')=0$ for all $x'\in C(x)$ and $\beta (y) =0$ at each $y \in G_{x'}$. \\

The above lemma is a consequence of the following proposition.
\begin{proposition}\label{prop: L's coincide}
 Let $L_1, L_2$ be two affine functions on $\R^d$, and let $S_1, S_2$ be sets in $\R^d$.
 Suppose that $L_1 \le L_2$ on $S_1$, and $L_2 \le L_1$ on $S_2$, and that  $IC(S_1)\cap IC(S_2) \ne \emptyset$.   
  Then, $L_1=L_2$ on $V(S_1 \cup S_2)$, the latter is the minimal affine space containing the sets $S_1$ and $S_2$. 
\end{proposition}
\begin{proof}
 This follows from two facts:
\begin{enumerate}
 \item  For affine functions,  $L \le L'$ on a set  $S$ implies $L_1 \le L_2$ on ${\rm conv}(S)$. 
 \item If two affine functions $L, L'$ satisfy $L \le L'$ on a set $S$ and if moreover, $L (z) = L'(z)$ at some interior point of ${\rm conv}(S)$, then $L= L' $ on ${\rm conv}(S)$, thus on $V(S)$. 
\end{enumerate}
Indeed, apply (1) to the case $L=L_1$, $L'=L_2$, and $S=S_1$, and also to the case 
 $L=L_2$, $L'=L_1$, and $S=S_2$. We get $L_1 = L_2$ on ${\rm conv}(S_1)\cap {\rm conv}(S_2)$. 
 Now, from the assumption, $IC(S_1)\cap IC(S_2) \ne \emptyset$,    and also obviously $IC(S_1)\cap IC(S_2) \subset IC(S_i)$,  
 $i=1,2$. Using (2) we then get that  $L_1 = L_2$ on both ${\rm conv}(S_i)$, $i=1,2$. From this the assertion follows. 
\end{proof}

\begin{proof}[Proof of Lemma~\ref{lem: trivial c}]
First note that for each $x, x'\in X=X_G$, conditions  \eqref{sd1} and \eqref{sd2} yield that $L_{x'} \le L_x$ on $G_x$, and $L_x\le L_{x'}$ on $G_{x'}$. \\
Now to prove the lemma, it suffices to show that for each $n \in \{ 0, 1, 2, ..., \}$,  
\begin{align}\label{n Lx}
 \hbox{if $x\sim_n x'$  then $L_x = L_{x'}$ on $V(C_n(x))$}.
\end{align}
Here, by $x\sim_0 x'$ we mean $x=x'$. 
We do this inductively. Our induction hypothesis is 
\eqref{n Lx} together with
\begin{align}\label{L less}
 L_z\le L_x \quad \hbox{on $C_n(x)$, for each $z \in X$ and $x \in X$}.
\end{align}
For $n=0$, \eqref{n Lx} is trivially satisfied and \eqref{L less} follows from \eqref{sd1} and \eqref{sd2}. Now, assume that \eqref{n Lx} and \eqref{L less} hold for all $n\le k$.
For $n=k+1$, if $x\sim_{k+1} x'$ then there are $x=x_0, x_1, ..., x_m=x'$ for some $m$,  such that $C_k(x_i) \cap C_k(x_{i+1})\ne \emptyset$ for each $0\le i \le m-1$. 
From this, and using \eqref{n Lx} and \eqref{L less}, we can apply Proposition~\ref{prop: L's coincide} with the choice $L_1=L_{x_i}, L_2=L_{x_{i+1}}$, $S_1=C_k(x_i)$, $S_2=C_k(x_{i+1})$, and see  $L_{x_i}= L_{x_{i+1}}$ on $V(C_k(x_i)\cup C_k (x_{i+1}))$, for $i=0, ..., m-1$. Similarly, repeated application of Proposition~\ref{prop: L's coincide} eventually yields that $L_x = L_{x_i}=L_{x'}$  on each $C_k(x_i)$, for $i=1, .., m$. 
Therefore, 
$L_x = L_{x'}$ on $\bigcup_i C_k(x_i)$, thus, on  $ V(\bigcup_{i=0}^m C_k(x_i))$.  
This holds for any $x \sim_{k+1} x'$, thus by applying the result to all $z \sim_{k+1} x \sim_{k+1} x'$, we also see 
\begin{align*}
\hbox{$L_x = L_{x'}$ on $V(\bigcup_{z \sim_{k+1} x} C_k(z)) = V(C_{k+1}(x))$,}
\end{align*} verifying \eqref{n Lx} for $n=k+1$. For \eqref{L less}, for each $z \in X$, from the assumption \eqref{L less} for $n\le k$ and applying \eqref{n Lx}, we have $L_z \le L_{x'}=L_x$ on $C_k (x')$ for all $x'\sim_{k+1} x$. For the affine functions, this implies $L_z \le L_x$ on $C_{k+1}(x)$. 
This completes the induction argument, so the proof. 
\end{proof}
We now consider the case of a non-trivial cost $c$. 
 We first establish a more quantitative version of Proposition~\ref{prop: L's coincide}. 
 \begin{proposition}\label{prop: quant L}
 Let $L_1, L_2$ be two affine functions on $\R^d$, and let $S_1, S_2$ be sets in $\R^d$.
 Suppose that 
 \begin{itemize}
 \item 
$L_1 \le L_2 +\delta_1 $ on $S_1$, and $L_2 \le L_1 + \delta_2$ on $S_2$ for some constants $\delta_1, \delta_2 >0$;
\item 
there is a point $z$ in $IC(S_1)\cap IC(S_2)$.  
\end{itemize}
  Then, 
$|L_1 -L_2| \le C$ on ${\rm conv}(S_1 \cup S_2).$
 Here,  $C=C(z, S_1, S_2, \delta_1, \delta_2) <\infty$ as long as $z$ stays in the interior $IC(S_1)\cap IC(S_2)$, though as $z$ gets close to the boundaries $\partial \left({\rm conv}(S_i)\right)$, $i=1$ or $2$,  the constant $C$ may go to $+\infty$. 
 \end{proposition}
\begin{proof}
 First, convexity and linearity imply that for each $\delta, \delta' >0$ we have the following:
\begin{enumerate}
 \item  For affine functions,  $L \le L'+\delta$ on a set  $S$ implies $L \le L' +\delta$ on ${\rm conv}(S)$. 
 \item If two affine functions $L, L'$ satisfy $L \le L' +\delta$ on a set $S$ and if moreover, $L(z) \ge L'(z) - \delta'$ at some interior point $z$ of ${\rm conv}(S)$, then $|L- L'|\le C=C(z, S, \delta, \delta') $ on ${\rm conv}(S)$. Here, the constant $C< \infty$ depends only on $\delta, \delta'$ and the ratio between the minimum distance from $z$ to $\partial \left( {\rm conv}(S)\right)$ and the maximum distance to $\partial ({\rm conv}(S))$, though, as $z$ gets close to $\partial \left({\rm conv}(S)\right)$, the constant $C$ can go to $+\infty$. 
\end{enumerate}
Now, apply (1) to the case $L=L_1$, $L'=L_2$, and $S=S_1$, and also to the case 
 $L=L_2$, $L'=L_1$, and $S=S_2$. Thus, we  get $|L_1 - L_2|\le \max(\delta_1, \delta_2)$ at the  point $z$ of $IC(S_1) \cap IC(S_2)$. Now, apply (2), to get 
 $|L_1 - L_2| \le C$ on both ${\rm conv}(S_i)$, $i=1,2$, where $C=C(S_1, S_2, \delta_1, \delta_2)< \infty$.  Applying (1) again, we have $|L_1-L_2| \le C $  on ${\rm conv}(S_1 \cup S_2)$, completing the proof.  
\end{proof}
 From now on, we consider only the maximization case, since the minimization case is the same by replacing $c(x, y)$ with $-c(x, y)$.
 We now introduce the following notation.
\begin{definition}\label{def: L H new}
 Let $G \in \mathcal{S}_{MT}$ and let $H \subset G$ be a $c$-contact layer for a triplet $\{\alpha, \gamma, \beta\}$. For each $x \in X_H$, consider the affine function
 \begin{align*}
&L_x^H(y) = \gamma(x) \cdot (y-x) + \alpha(x).
\end{align*}
The superscript $H$ indicates that $L_x^H$ arises from a $c$-admissible triplet for $H$. 
The fact that $H$ is a $c$-contact layer for $\alpha, \gamma, \beta$ can be written as:
\begin{align}
\label{sd1 quant}
\beta(y)- c(x, y) &\geq L^H_x(y),\, \forall x\in X_H, y \in Y_H,  \\
\label{sd2 quant}
\beta(y)- c(x, y)  &= L^H_x(y), \,\, \forall (x,y) \in H.
\end{align}
For an affine space $V$, we write for $x, x' \in X_H$, 
\begin{align*}
L_x^H \approx  L_{x'}^H \,\,\text{ on $V$}
\end{align*}
if there is a bounded  set $S$ with $V=V(S)$ and a constant $M=M(c, H, S)$ depending only on $H$, the cost function $c$ and the set $S$,  such that for every choice of a $c$-admissible triplet $\{\alpha, \gamma, \beta\}$ making $H$ a $c$-contact layer,   
 we have
 \begin{align}\label{apriori new}
|L_x^H - L_{x'}^H |\le M \quad \hbox{on the set $S$.}
\end{align}
We say $L_x^H \approx  L_{x'}^H$ at $z$, if  we have \eqref{apriori new} for $S =\{z\}$. 
\end{definition}
\noindent An immediate observation is that for $x, x', x'' \in X_H$, 
\begin{align*}
 \hbox{whenever  $L_x^H \approx L_{x'}^H$ and $L_{x'}^H \approx L_{x''}^H$ on $V$,  then $L_{x}^H \approx L_{x''}^H$ on $V$.}
\end{align*}
Also note that if $H' \subset H$,  
we necessarily have for any $x, x' \in X_{H'}$, 
\begin{align}\label{3.23}
L_x^{H'} \approx  L_{x'}^{H'}  \quad \hbox{on $V$} \Rightarrow L_x^H \approx  L_{x'}^H   \quad \hbox{ on $V$}.
\end{align}
We shall now prove the analogue of Lemma~\ref{lem: trivial c} in the case of a general cost.
We shall again use Proposition~\ref{prop: quant L}, 
to 
establish a propagation of control on the affine functions $L_x^H$'s, along an ordered chain of intersecting convex open sets.  But, since $c$ is not trivial anymore, the control on $L_x^H$ can be done only in finite steps, since the errors (the constant $C$ in Proposition~\ref{prop: quant L}) can accumulate.

\begin{lemma}\label{3.27}
Set $G:=\Gamma \cap (C(x_0) \times \R^d)$ and suppose $x, x' \in X_G$ (i.e. $x \sim  x'$).   Then there exists a finite set $H \subset G$  such that  $x, x' \in X_H$  
 and $L_x^H\approx L_{x'}^H$ on $V(C(x))$.  
\end{lemma}
\begin{proof}
\noindent First observe that it suffices to prove
\begin{claim}\label{claim: propagation}
 Suppose $x \sim  x'$ and $z \in G_{x'}$.   Then there exists a finite set $H \subset G$  such that  $x, x' \in X_H$, and $L_{x}^H\approx L_{x'}^H$ at $z$. 
\end{claim}
The lemma follows when we apply this claim to a set of finitely many $(u_i, v_i)$'s, $i=1,..., m$,  in $G$ (so $x \sim x' \sim u_i$)  with $V(C(x)) = V(\{v_i\}_{i=1}^m)$, and use \eqref{3.23}.\\
We show this claim using  induction on $n=0,1,2,3, ... $. 
Our induction hypothesis is   if $x\sim_n x'$ and $z \in G_{x'}$, then there exists a finite set $H \subset G$ such that 
\begin{itemize}
\item[(i1)] $x, x' \in X_H$, $z \in Y_H$  and $Y_H \subset \overline{C_n(x)}$;
\item[(i2)] 
 $L^H_x \approx L^H_{x'}$  on $V(Y_H)$;
 \item[(i3)]  for each  finite set $F \subset G$ with $H \subset F$, and for $w \in X_{F}$, there is a constant $C=C(H,w)$ depending only on $H$ and $w$ such that  
\begin{align*}
 L_w^{F}\le L_x^{F} + C \quad \hbox{on $Y_H$}. 
\end{align*}
\end{itemize}
Notice that the claim follows if we verify (i1) -- (i3) for each $n$, since in particular, the values of $L_x^H$ and $L_{x'}^H$ at $z$ is estimated from (i2). \\
We proceed by induction, starting with $n=0$ and assuming $x\sim_0 x'$ (i.e. $x=x'$) and $z \in G_{x'}$.  Choose then $H=\{(x, z)\}$. Then, (i1) and (i2) are trivially satisfied.  
Moreover,  (i3) holds from \eqref{sd1 quant} and \eqref{sd2 quant}, where the constant $C$ is estimated by the value $c(w, z) - c(x, z)$. 
  This completes the case $n=0$. \\
Suppose now the induction hypothesis holds for $n$.
Assume $x\sim_{n+1}x'$ and $z \in G_{x'}$. Then, there is a finite chain of $C_n(x_k)$, $k=0,1,..., m$, $x_0=x$, $x_{m}=x'$, 
such that $C_n (x_k) \cap C_n (x_{k+1}) \ne \emptyset$ for each $k$.   Recall $C_n(x) =  IC[\bigcup_{x' \sim_n x} G_{x'}]. $ Thus for each $C_n(x_k)$, it is possible to  find a finite set $J_k:= \{ (u^i_k, v_k^i)_{i=1}^{m_k} \}\subset G$ such that  $x_k \sim_n u^i_k$ for all $i$ and $IC(Y_{J_k})$ is a good approximation of $C_n(x_k)$, i.e.
$Y_{J_k} \subset \overline{C_n (x_k)} \subset V(Y_{J_k})$ and $IC(Y_{J_k}) \cap IC(Y_{J_{k+1}}) \ne \emptyset$. Also, we can let $z \in Y_{J_m}$.\\
Now, apply the induction hypothesis for $n$ to each $x_k, u_k^i$, $x_k \sim_n u_k^i$ and find a finite set $H^i_k \subset G$ that satisfies (i1)--(i3) for $x=x_k$, $x'=u^i_k$, $z=v^i_k$, and $H=H^i_k$.   
Let
\begin{align*}
H_k : = \bigcup_i H^i_k 
\end{align*}
Then, from (i3) for $H^i_k$'s, we also have (i3) for $H=H_k$ and $x=x_k$.
Here, the point of considering $H_k$ is  $Y_{J_k} \subset Y_{H_k} \subset \overline{C_n(x_k)}$, so $V(Y_{J_k}) = V(Y_{H_k})$, hence $IC(Y_{H_k}) \cap IC(Y_{H_{k+1}}) \ne \emptyset$ as well. \\
In order to verify the induction hypothesis for $n+1$-th step,  let 
\begin{align*}
\bar H: = \bigcup_k H_k 
\end{align*}
We will show properties (i1)--(i3) for this set $\bar H$. 
From the construction, $x, x' \in X_{\bar H}$, $z \in Y_{\bar H}$, and since $C_n(x_k) \subset C_{n+1}(x_k) = C_{n+1}(x)$, (i1) readily follows.
For (i2), 
apply the induction hypothesis (i3) for $H_k$'s to Proposition~\ref{prop: quant L} iteratively for the pairs $Y_{H_1}$ and $Y_{H_2}$, $Y_{H_1} \cup Y_{H_2}$ and $Y_{H_3}$, ... , $Y_{H_1} \cup .... \cup Y_{H_k}$ and $Y_{H_{k+1}}$, so on. 
Then we see the estimate \eqref{apriori new} holds for $S=Y_{\bar H}$, thus,   
\begin{align}\label{x1}
 L^{\bar H}_x \approx L^{\bar H}_{x_1} \approx...\approx L^{\bar H}_{x_{m-1}} \approx L^{\bar H}_{x'} \quad \hbox{ on $V(Y_{\bar H})$}, 
\end{align}
verifying (i2). \\
For (i3), let $F$ be a finite set containing $\bar H$ and let $w \in X_F$. Then (i3) for each $H_k$ gives that $L^F_w \leq L^F_{x_k} + C_k$ on $Y_{H_k}$, $k=0,1,...,m$. Now applying \eqref{x1} and recalling \eqref{3.23}, we conclude that there is a constant $C=C(\bar H,w)$ such that
$$L^F_w \leq L^F_{x} + C\,\, \text{ on } \,\,Y_{\bar H}.$$   
This completes the induction, and the proof.
\end{proof}
\begin{proof}[\bf Proof of  Lemma \ref{3.14}]
 Recall that we fix $x_0 \in X$ and let 
$G:=\Gamma \cap (C(x_0) \times \R^d)$ and $V:= V(Y_G)$.  
Let $m = dim (V)$. Then we can find
\begin{align*}
J := \{ (u_i, v_i) \} _{i=0} ^m \subset G \,\, \text{ such that } \,\,
V(\{v_i \}_{i=0} ^m ) = V.
\end{align*}
Define the initial choices $K_{v_i} = \{0\}, i = 0,1,...,m$.  
We want to define the $K_y$'s to be compatible with these initial choices.
For $y \in Y_G$, choose $x(y) \in X_G$ such that $(x(y),y) \in G$. By lemma \ref{3.27} (especially see Claim~\ref{claim: propagation}), for $y \in Y_G$, we can choose a  finite set  $H(y)$ such that $J \cup \{(x(y),y)\} \subset  H(y)$ and 
\begin{align*}
 \hbox{$L_{x(y)} ^{H(y)} \approx L_{u_i}^{H(y)}$ at  $v_i$,  $\forall i=0, ..., m$}.
\end{align*}
In particular, there exists a constant $M$, depending only on $y$ and $H(y)$ --but not on the choice of the $c$-admissible functions for which $H(y)$ is a $c$-contact layer-- such that 
\begin{align}\label{l}
| L_{x(y)}^{H(y)} (v_i) - L_{u_i}^{H(y)} (v_i) | \leq M, \,\,\forall \, i = 0,...,m.
\end{align}
 $H(y)$ being a $c$-contact layer for some triplet  $(\alpha, \gamma, \beta)$, we can  
by subtracting an appropriate affine function from $\beta$, assume $\beta (v_i) = 0$.
This yields that 
\begin{align*}
\beta (y) = c(x(y) , y) + L_{x(y)}^{H(y)} (y).
\end{align*}
Since $ L_x^{H(y)}$ is affine and $V(\{v_i \}_{i=0} ^m ) = V$, the value $L_{x(y)}^{H(y)} (y)$ can be computed from the values 
$L_{x(y)}^{H(y)} (v_i)$. Hence by (\ref{l}), the values $L_{u_i}^{H(y)} (v_i)$ give an estimate of $\beta(y)$. Notice that the $c$-contact property yields that 
\begin{align*}
L_{u_i}^{H(y)} (v_i) = \beta (v_i) - c (u_i, v_i) = - c (u_i, v_i).
\end{align*}
Thus, there exists a constant $N=N(y)$ such that if $\beta$ is a $c$-admissible for $H(y)$ and if $\beta (v_i) = 0$ for all $i$, then $-N \leq \beta(y) \leq N$. We set $K_y = [-N, N]$. \\
To get the claim in Lemma \ref{3.14}, we let $H$ be any finite set and denote $Y_H = \{y_1,...,y_s\}$. Let 
\begin{align*}
H^* = H \cup H(y_1) \cup ... \cup H(y_s)
\end{align*}
Now choose $\beta$ to be $c$-admissible for $H^*$ with $\beta(v_i) =0$ for all $i$. Since $\beta$ is also $c$-admissible for $H(y_j)$, we have $\beta(y_j) \in K_{y_j}$ for all $j=1,...,s$. Finally, note that $\beta$ is also a $c$-admissible for $H$, concluding the proof.
\end{proof}

\noindent{\bf Proof of Theorem \ref{local.layers}:}
As before, let $ G=\Gamma \cap (C(x_0) \times \R^d)$ and  let $V := V(Y_G)$ be the ambient space. We first find  the desired function $\beta: Y_G \to \R$ from the compactness argument already used in \cite{bj}. 
Indeed, define $K := \Pi_{y \in Y_G} K_y$, where the $K_y$'s were obtained in Lemma \ref{3.14}. This is a subset of the space of all functions from $Y_G$ to $\R$. In the topology of pointwise convergence, $K$ is compact by Tychonoff's theorem. Now we claim that, for any finite $H \subset G$, the set
\begin{align*}
\Psi_H := \{ \beta \in K : \beta \text{ is $c$-admissible for $H$}\}
\end{align*}
is a non-empty closed subset of $K$. Indeed, that $\Psi_H$ is non-empty follows from  Lemma \ref{3.14} since every finite subset of  $\Gamma$ and hence of $G$ is a $c$-contact layer: if  necessary, one can extend the $\beta$ found in Lemma~\ref{3.14} --and originally defined on $Y_H$-- to $Y_G$, by simply letting $\beta(y) =0$ for $y\notin Y_H$. \\
  To show that $\Psi_H$ is closed, let $\{ \beta_n \}$ be a sequence of $c$-admissible functions for $H$, and suppose $\beta_n \to \beta$ pointwise on $Y_G$. 
We need to show that $\beta$ is also $c$-admissible for $H$. But for each $n$, we have functions $(\alpha_n, \gamma_n)$ such that the following relation holds:
\begin{align}
\label{m}\beta_n (y)-c(x,y) &\geq \gamma_n(x) \cdot (y-x) + \alpha_n(x)\,\, \forall x\in X_H, y \in Y_H \\
\label{n}\beta_n(y)-c(x,y) &= \gamma_n(x) \cdot (y-x) + \alpha_n(x)\,\, \forall (x,y) \in H.
\end{align}
Here, without loss of generality, we can assume that each vector $\gamma_n(x)$ is parallel to $V(Y_H)$.
 Now since $(\beta_n (y)-c(x,y))_{x \in X_H, y \in Y_H} $ is uniformly bounded in $n$,  we can choose $(\alpha_n(x), \gamma_n(x))$ in such a way that  $(\alpha_n(x), \gamma_n(x))_n$ is also uniformly bounded in $n$. Since $X_H$ is finite, we can find a subsequence of $(\alpha_n, \gamma_n)$ which converges to 
$(\alpha, \gamma)$ at every $x \in X_H$. Then $(\alpha, \gamma, \beta)$ is clearly a $c$-admissible triplet  for $H$, establishing the claim on $\Psi_H$.\\
It is clear that the class $\{\Psi_H \}$ satisfies the finite intersection property, that is  $ \emptyset \ne \Psi_{H_1 \cup...\cup H_s} \subset \bigcap_{j=1,...,s} \Psi_{H_j}$.
By the compactness of $K$ and the closeness of $\Psi_H$'s, we deduce that the set $\Psi_G := \bigcap_{H \subset G, |H|<\infty}  \Psi_H$ is nonempty. \\
We now claim that any $\beta \in \Psi_G$ is $c$-admissible for $G$. Indeed, 
fix $x\in X_G$ and  $\beta \in \Psi_G$. We must show that there exists an affine function $L_x$ on $V=V(Y_G)$ such that the following holds:
\begin{align}
\label{r}\beta(y)-c(x,y) &\geq L_x(y), \,\, \forall y \in Y_G \\
\label{s}\beta(y)-c(x,y) &= L_x(y), \,\, \forall y \in G_x.
\end{align}
Choose a finite set $H_x \subset G_x$ such that $V(H_x) = V(G_x)$. Observe that for any finite set $F$ containing $H:=\{x \} \times H_x$, 
\begin{align*}
L_x^F(y) = \beta (y) - c(x,y) = L_x^H (y) \,\,\forall y \in H_x, \text{ hence } L_x^F(y) = L_x^H (y)\,\,\forall y \in V(G_x).
\end{align*}
In particular, $L_x^F(x) = L_x^H (x)$ since $x \in IC(G_x) \subset V(G_x)$. Let us define $\alpha(x) = L_x^H (x)$.\\
Now we need to construct the last piece which is $\gamma(x)$.  
 For this, in addition to $H$, we also choose a finite set $\{ (v_i, w_i) \}^m_{i=1} \subset G$ such that $x \in IC( \{w_i\} ^m_{i=1})$ and $V( \{w_i\}^m_{i=1}) = V$, and 
 define
\begin{align*}
\bar H := H \cup \{ (v_i, w_i) \}^m_{i=1}.
\end{align*}
For any finite set $F \subset G$ with $\bar H \subset F$,  define the set
\begin{align}
\label{p}\gamma_F (x) := \{v \in V : \beta(y) - c(x,y) -\alpha (x) &\geq v \cdot (y-x), \,\, \forall y \in Y_F\\
\label{q} \beta(y) - c(x,y) -\alpha (x) &= v \cdot (y-x), \,\, \forall y \in F_x \}
\end{align}
 The set $\gamma_F(x)$ is nonempty because $G$ itself is a $c$-contact layer. 
Now, since $x \in IC( \{w_i\} ^m_{i=1})$ and 
$V( \{w_i\}^m_{i=1}) = V$, we deduce from (\ref{p}) that $\gamma_F (x)$ is a closed and bounded set %
 in $V$, hence compact. Again since every subset of a $c$-contact layer is also a $c$-contact layer, 
the class
$\{ \gamma_F (x) : F \supset \bar H \} $ has the finite intersection property. Hence, we can choose a 
\begin{align*}
\gamma(x) \in \bigcap_{F \supset \bar H, |F|<\infty}  \gamma_F(x).
\end{align*}
Finally, we show that (\ref{r}) and (\ref{s}) hold for this choice of $(\alpha(x), \gamma(x))$. Indeed, 
 let $(x', y') \in G$, and let
 $F = \bar H \cup \{(x', y')\}$. By (\ref{p}), we have 
$\beta(y') - c(x,y')  \geq \alpha (x) + \gamma(x) \cdot (y'-x)$, so (\ref{r}) holds.
Let $y\in G_x$. Let $F = \bar H \cup \{(x, y)\}$. By (\ref{q}), we have $\beta(y) - c(x,y) = \alpha (x) + \gamma(x) \cdot (y-x)$, so (\ref{s}) holds. This completes the proof of Theorem~\ref{th:k}. \qed

\section{Structural results for general optimal martingale transport plans}\label{s: structure}
We start by proving Conjecture 2) in the case of a discrete target measure. 
\begin{theorem} 
\label{th: finite images2}
Let $c(x,y)=|x-y|$, suppose $\mu << \mathcal{L}^d$ and that $\nu$ is discrete,  i.e. $\nu$ is supported on a countable set.  Let $\pi\in MT(\mu, \nu)$ be a solution of \eqref{MGTP}, then for $\mu$ a.e. x, $\supp \pi_x$ consists of $d+1$ points which are vertices of a polytope in $\R^d$, and therefore the optimal solution is unique.
\end{theorem}

\begin{proof} Since the result holds true (for more general target measures) when $d=1$, we shall assume that $d\geq 2$. Let $S$ be the countable support of $\nu$ and  let $J:=\{E \subset S : |E| < \infty \quad \& \ \  \dim V(E)  \leq d-1\}$, where $|E|$ is the cardinality of the set $E$. Consider $V_J := \cup_{E \in J} V(E)$. Since  $\dim V(E) \leq d-1$ and $J$ is countable, it follows that $\mathcal{L}^{d}(V_J)=0$.  Let $\Gamma$ be a martingale-monotone regular concentration set for $\pi$ (as in Defintion~\ref{def: mg mon}). Let $X:= X_\Gamma \setminus V(J)$ so that $\mu(X)=1$. Now notice that if $x \in X$, then $\Gamma_x$ must contain vertices of a polytope which has $x$ in its interior.\\
Let now $K := \{E \subset S : \text{ $|E| = d+2$ and $E$ contains vertices of a $d$-dimensional polytope} \}$. 
Fix $F =\{y_0, y_1,...,y_d, y\}$ in $K$, where $y_0, y_1,...,y_d$ are vertices of a  $d$-dimensional  polytope and consider the set
$A:= \{x \in X : F \subset \Gamma_x \}.$
In other words, $A = \Gamma^{y_0} \cap ... \cap \Gamma^{y_d} \cap \Gamma^{y}$, where 
$\Gamma^{y} := \{x : (x,y) \in \Gamma \}$. We shall prove that $\mu(A) =0$. \\
Indeed, suppose otherwise, that is $\mu(A) >0$ and let $x_0$ be a Lebesgue point of $A$. Let $B = A \cap C(x_0)$ and note that $\mathcal{L}^d(B) >0$ since $C(x_0)$ is open in $\R^d$. Since the set $\Gamma \cap (C(x_0) \times \R^d)$ is a $c$-contact layer, there exist constants $\lambda_0, \lambda_1,..., \lambda_d, \lambda$ such that for all $x \in B$, we have
\begin{align*}
|x - y_i| + \gamma(x) \cdot (y_i-x) + \alpha(x) &= \lambda_i, \,\,i=0,1,...,d\\
|x - y| + \gamma(x) \cdot (y-x) + \alpha(x) &= \lambda.
\end{align*}
Also note that $\{y_0, y_1,...,y_d, y\} \subset {\rm Ext}\,\big{(} {\rm conv}(\Gamma_x) \big{)}$ for almost all $x\in B$. Let $p_i$ be determined by $y=\Sigma^d _{i=0} p_i y_i$, and $\Sigma^d _{i=0} p_i=1$, and note that some $p_i$ may be negative. Then, by the above, we get that the function 
\begin{align*}
 g(x) := \Sigma^d _{i=0} p_i |x - y_i| - |x - y| 
\end{align*}
is constant on $B$, which has positive measure. \\
We explain why this leads to a contradiction. First, notice that because $g$ is real analytic in $\Omega:=\R^d \setminus \{ y_0, ..., y_d, y\}$, it is not constant in any open subset, since otherwise it is constant everywhere, which is not the case. 
Second, without loss of generality, assume $x_0=0$ and $g(0)=0$, 
and notice that from the real analyticity of $g$,  one can write $g(x) = P_k(x) + Q(x)$ for some  $k\in \N$,  where $P_k(x)$ is  the first nonzero $k$-th degree homogeneous polynomial, and $Q(x)$ is a power series of terms with degree greater than $k$, in particular, $Q(x)  = O(|x|^{k+1})$.
Now, consider the set 
\begin{align*}
{ \mathcal S}:=\{ u \in S^{d-1} \ | \ \hbox{ there exists $0\ne x_n \to 0$, $x_n / |x_n| \to u$, with $0=g(x_n)$}\}.
\end{align*}
Then, for each $u \in { \mathcal S}$, 
$0=\frac{  g(x_n)}{|x_n|^k} = P_k(x_n/|x_n|) +  \frac{Q(x_n)}{|x_n|^k}$, showing
 $P_k (u) = \lim_{n\to \infty} P_k(x_n/|x_n|) =0$. Thus, ${ \mathcal S}$ is a subset of the zero set $\{ u;\, P_k(u) =0\}$.\\
Now if $g$ is zero on the set $B$ where $x_0$ is a Lebsegue point, then ${ \mathcal S} = S^{d-1}$, hence $P_k = 0$, a contradiction. Hence, $\mu(A)=0$. The countability of $K$ now implies the theorem.\\
For the uniqueness, we use the usual argument, namely that the average of two optimal plans is also optimal, which contradicts the polytope-type of their respective supports. 
\end{proof}

\begin{remark}
 As we see from the above proof, Theorem~\ref{th: finite images2} holds true for a much more general  cost $c(x, y)$ than $|x-y|$. Indeed, it is enough (but not necessary) $c(x, y)$ to be analytic in $\{x \ne y\}$, and the function $g(x) = \sum_{i=0}^d p_i c(x, y_i) - c(x, y)$ to be  non-constant.  In particular, we can choose $c(x, y) = |x-y|^p$, with $p\ne 2$. 
\end{remark}

We now establish Conjecture 1) in the 2-dimensional case.

\begin{theorem} 
\label{th: Choquet 2d}
Assume $d=2$, $c(x,y) = |x - y|$, $\mu$ is absolutely continuous with respect to the Lebesgue measure, and  $\nu$ has compact support. Let $\pi\in MT(\mu, \nu)$ be a solution of \eqref{MGTP}, then  for $\mu$ almost every $x\in \R^2$, ${\rm supp}\, \pi_x = {\rm Ext}\,\big{(} \overline{{\rm conv}}({\rm supp}\, \pi_x) \big{)}$.
\end{theorem} 

\begin{proof}    Let $\Gamma$ be a martingale-monotone regular concentration set for $\pi$  (see Lemma~\ref{treatment} (4) and Defintion~\ref{def: mg mon}), and let $X = X_\Gamma$. (Recall then $\supp \pi_x = \overline{\Gamma_x}$ for all $x\in X$.)  The theorem will follow if we show that the set
\begin{align*}
 E_\pi :=\{ x \in  X \ |  \  {\rm supp \, } \pi_x \subset {\rm Ext} ( \overline{{\rm conv}} ({\rm supp \, } \pi_x ))\} 
\end{align*} 
has full $\mu$-measure. First note that $E_\pi$ is measurable 
by Proposition \ref{prop:measurable E}. (Here, we used the fact that 
   each of  ${\rm supp\,} \pi_x \subset \R^d $  is compact, which  is satisfied since the second marginal of $\pi$ is compactly supported.)



 We shall show that its complement $N=X\setminus E$ has $\mu$-measure zero  and since $\mu << \mathcal{L}^2$ it suffices to show that $\mathcal{L}^2 (N)=0$. For that note first that the set  $X_0:=\{ x\in X:{\rm dim(conv}(\Gamma_x)) = 0 \}$ 
 is obviously included in  $E$, which means that $N= (N\cap X_2)\cup (N\cap X_1)$, where 
$$X_2 = \{x \in X : \dim V(C(x)) = 2\}\quad \hbox{and \quad $X_1 = \{x \in X : \dim V(C(x)) = 1\},$}$$
 where $\{C(x); x\in X\}$ is the irreducible convex paving of $\Gamma$. \\
 Note that $X_2 = \cup_{x \in X_2} (X \cap C(x)) = X \cap (\cup_{x \in X_2} C(x) )$.  Since $\cup_{x \in X_2} C(x)$ is open, $X_2$ is measurable. 
But since $\Gamma \cap (C(x) \times \R^2)$ is a $c$-contact layer, Theorem \ref{when.duality} yields that $\overline{\Gamma_x} = {\rm Ext}\,\big{(}{\rm conv}(\overline{\Gamma_x}) \big{)}$ for a.e. $x$ in $X_2 \cap C(x)$. Since $X_2$ can be approximated by compact sets from the inside and $\{C(x)\}_{x \in X_2}$ is an open cover of $X_2$, we conclude that  $\overline{\Gamma_x} = {\rm Ext}\,\big{(}{\rm conv}(\overline{\Gamma_x}) \big{)}$  for a.e. $x$ in $X_2$.
Hence,  $\mathcal{L}^2 (N\cap X_2)=0$.\\
Consider now the 
 measurable set $A_1:=N\cap X_1$, and assume that 
 $\mathcal{L}^2 (A_1)>0$.
Note that for every $x\in A_1$, we have that 
$I(x) := IC(\supp \pi_x)$ is an open line segment with $x$ in its interior.   Note that  $I(x) \subset C(x)$ and $C(x)$ is one-dimensional for every $x \in A_1$.  By Proposition~\ref{prop:meas w}, the function defined for each $x\in A_1$ by
$$
\delta(x)=\sup\{r;  (x - r , x+r) \subset I(x)
$$
is measurable, 
 where $(x - r , x+r)$ denotes the interval of radius $r$ at $x$ inside the line segment $I(x)$. 
Therefore,  the set
$A_\delta := \{x \in A_1 : \delta(x) > \delta\}$ for every $\delta>0$ is also measurable, and 
 $\mathcal{L}^2 (A_\delta)>0$
for some $\delta >0$. Let now $x_0$ be a Lebesgue point of $A_\delta$, and consider $W$ to be the 1-dimensional affine space containing $x_0$ and perpendicular to $I(x_0)$. Choose $\epsilon >0$ much smaller than $\delta$ and let $A_{\delta, \epsilon} := A_{\delta} \cap B(x_0,\epsilon)$  (note $\mathcal{L}^2 (A_{\delta, \epsilon})>0$). Then 
$\big\{C(x); x \in A_{\delta, \epsilon}\big\}$ is a disjoint family of open segments that cover $A_{\delta, \epsilon}$ and $C(x) \cap W \neq \emptyset$. Let $F : \bigcup_{x \in A_{\delta, \epsilon}} C(x) \to \bigcup_{x \in A_{\delta, \epsilon}} F(C(x)) $ be the flattening map with respect to $W$ as in Lemma \ref{4.4}. Since $F$ is bi-Lipschitz on the appropriate set containing $A_{\delta, \epsilon}$, we have that 
  $F(A_{\delta, \epsilon})$ is measurable and $\mathcal{L}^2 (F(A_{\delta, \epsilon})) >0$.\\
Note that again by Theorem \ref{when.duality},  $\overline{\Gamma_z} = {\rm Ext}\,\big{(}{\rm conv}(\overline{\Gamma_z}) \big{)}$, 
for $\mathcal{L}^{1}$ almost all $z$ in each $A_{\delta, \epsilon} \cap C(x)$. Since  $A_{\delta, \epsilon} \subset N$, this implies that $A_{\delta, \epsilon} \cap C(x)$ is $\mathcal{L}^1$ measure zero, and so does $F(A_{\delta, \epsilon}) \cap F(C(x))$. Now $\big\{F(C(x));  x \in A_{\delta, \epsilon}\big\}$ is a parallel cover of $F(A_{\delta, \epsilon})$, so by Fubini's theorem with bi-Lipschitz map $F$,  we conclude 
$\mathcal{L}^2 (F(A_{\delta, \epsilon}))=0$, which is a contradiction. 
(Here for the Fubini's theorem,  we used the fact that $F(A_{\epsilon, \delta})$ is measurable.)
 It follows that  
 $\mathcal{L}^2 (A_1)=0$, which then results  
  $\mathcal{L}^2 (N)=0$.
This completes the proof. 
\end{proof}

The same proof could extend to higher dimensions, 
provided one can prove measurability of  the function 
\begin{align*}
X_\Gamma \ni x  \mapsto 
 \delta(x) = \sup \big\{r \geq 0 : B(x,r) \subset C(x)\big\} 
\end{align*}
defined 
 for a given convex paving  $(C(x))_{x\in X_\Gamma}$  associated to $\Gamma$.
One can then obtain the following. 
\begin{theorem} 
\label{th: Choquet codim 1}
Assume $c(x,y) = |x - y|$ on $\R^d\times \R^d$ and let $\pi \in MT(\mu, \nu)$ be a solution of \eqref{MGTP} with a martingale-monotone regular concentration set $\Gamma$.  Assume $\mu$ is absolutely continuous with respect to the Lebesgue measure and that 
\begin{align}\label{co1}
\hbox{the function $\delta$ is measurable, and  \,\, ${\rm dim}(V(C(x))) \geq d-1$ \,\, for $\mu$ a.e. $x$,}
\end{align}
where $(C(x))_{x\in X_\Gamma}$ is the irreducible convex paving associated to $\Gamma$. 
Then,  for $\mu$ almost every $x\in \R^d$, ${\rm supp}\, \pi_x = {\rm Ext}\,\big{(} {\rm conv}({\rm supp}\, \pi_x) \big{)}$
.\end{theorem}

\section{The disintegration of a martingale transport plan}\label{s:disintegration}
 
 For a closed convex set $U \subset \R^d$, let  $\mathcal{K}(U)$ be the space of all closed convex subsets in $\R^d$, equipped with the Hausdorff metric in such a way that it becomes a separable complete metric space (Polish space). 
This allows for the disintegration of a measure $\pi$ on $X$ via a measurable map $T: X \to \mathcal{K}(U)$ (see e.g. \cite[Corollary 2.4]{bc}) in such a way that each piece of the disintegrated measure, say $\pi_C$, is a probability measure on $T^{-1}(C)$. In particular, $\pi_C (T^{-1}(C) )=1$ for $T_\#\pi$-a.e. $C \in \mathcal{K}(U)$, ultimately  yields conditional probabilities. 

Consider now a set $\Gamma \in \mathcal{S}_{MT}$ and the corresponding unique irreducible convex paving  $\{C; C \in \Phi\}$  as given in Theorem~\ref{general.dec.}.  Define the map 
$$ \Xi: \Gamma \to \mathcal{K}(\R^d) \quad {\rm by}\quad (x,y) \mapsto \overline{C(x)},$$
 where $\mathcal{K}(\R^d)$ is the space of convex closed subsets of $\R^d$. We conjecture that this map is measurable when  $\mathcal{K}(\R^d)$ is equipped with the Hausdorff metric that makes a separable complete metric space. In this case, we 
 shall show  that a  martingale transport plan $\pi$ can be canonically disintegrated into its components given by $(\Gamma \cap (C(x) \times \R^d))_{x\in X_\Gamma}$. As usual, in the case of minimization with $c(x,y) = |x - y|$, we shall  assume further that $\mu \wedge \nu =0$.  
  
\begin{theorem}[\bf Disintegration of martingale plans]\label{th:disintegration} 
Let $(\mu, \nu)$ be probability measures on $\R^d$ in convex order and let $\pi \in MT(\mu, \nu)$ with a concentration set $\Gamma \in \mathcal{S}_{MT}$ and the associated irreducible convex paving  $\{C; C \in \Phi\}$. Assume the map 
$
 \Xi: \Gamma \to \mathcal{K}(\R^d)$ defined  by $ (x,y) \mapsto \overline{C(x)},
 $
is measurable, and let $\tilde \pi =\Xi_\# \pi$ denote the push-forward of $\pi$ into $\mathcal{K}(\R^d)$, and $I \subset \mathcal{K}(\R^d)$ is the image of $\Gamma$ by $\Xi$. Then the following holds:

\begin{enumerate}[(1)] 

\item \ There exists a disintegration of $\pi$ along the map $\Xi$ such that 
\begin{align}\label{eq: disintegration}
 \pi (S)=\int_{I} \pi_{C} (S) d\tilde\pi (C)  \quad \hbox{ for each Borel set  $S \subset \R^d \times \R^d$}, 
\end{align}
where for $\tilde \pi$-a.e. $C$, $\pi_C$ is a  probability measure supported on $\Gamma_C:=\Gamma \cap (C \times \R^d)$.
\item \ 
For $\tilde \pi$-a.e. $C\in I$, there exist probability measures  $\mu_C, \nu_C$ such that the couple  $(\mu_C, \nu_C)$ is in convex order,  $\mu_C$ is supported on  $X_C : = X_\Gamma \cap C$, $\nu_C$ on $Y_{\Gamma_C}$, and $\pi_C\in MT(\mu_C, \nu_C)$.

 \item \ 
  If $\pi$ is optimal for  problem (\ref{MGTP}) in MT($\mu, \nu$), then for $ \tilde \pi$-a.e. $C\in I$, $\pi_C$ is optimal for the same problem on $MT(\mu_C, \nu_C)$. Furthermore, $\Gamma_C$ is a $c$-contact layer. In particular, duality is attained for $\pi_C$. 

\item \ 
 If in addition,  $\mu_C$ is absolutely continuous with respect to the Lebesgue measure on $V(C)$, and $c(x,y) = |x - y|$, then 
for $\mu_C$-almost all $x$,  $\overline{\Gamma_x} = {\rm Ext} \, \big{(}{\rm conv}(\overline{\Gamma_x})\big{)}$. 
\end{enumerate}
\end{theorem}

 \begin{proof}
The above discussion and the measurability hypothesis  of the map $ \Xi: \Gamma \to \mathcal{K}(\R^d)$ defined by $ (x,y) \mapsto \overline{C(x)}$,  
yield the disintegration of $\pi$ into $\pi_C d\tilde \pi(C)$ in \eqref{eq: disintegration}, with $\pi_C$ supported on $\Gamma_C$. 
 The measures $\mu_C, \nu_C$ are obtained by taking marginals of $\pi_C$. The  martingale and optimality properties of $\pi_C$ for $\tilde \pi$-a.e. $C$, follow from those properties of $\pi$ and the disintegration \eqref{eq: disintegration}. When $\pi$ is an optimal martingale transport, the concentration set $\Gamma$ can be chosen in such a way that it is $c$-finitely exposable,  hence the set $\Gamma_C$ is a $c$-contact layer by Theorem~\ref{th:k}.
This deals with items (1),  (2), and (3) of the theorem.
 Finally, (4)  follows immediately from  Theorem~\ref{thm: main extremal}.  
 \end{proof}
 
In order to apply this theorem and deduce global results from its local properties, one would like to know when we can disintegrate $\mu$ into absolutely continuous pieces $\mu_C$, so as to apply  Theorem~
\ref{thm: main extremal} 
on each partition. We start by a counterexample showing that this is not possible in general, at least in dimension $d\geq 3$. \\

\noindent{\bf Nikodym sets and martingale transports:}
Ambrosio, Kirchheim, and Pratelli \cite{akp} constructed a Nikodym set in $\R^3$ having full measure in the unit cube, and intersecting each element of a family of pairwise disjoint open lines only in one point. More precisely they showed  the following.

\begin{theorem}{\rm (Ambrosio, Kirchheim, and Pratelli \cite{akp}) }
\label{akp}
There exist a Borel set $M_N \subset [-1,1]^3$ with $|[-1,1]^3 - M_N| = 0$ and a Borel map
$f =(f_1,f_2): M_N \to [-2, 2]^2 \times [-2, 2]^2$ such that the following holds. If we define for $x \in M_N$ the open segment $l_x$ connecting $( f_1(x), -2)$ to $( f_2(x), 2)$, then
\begin{itemize}
\item $\{x\} = l_x \cap M_N$ for all $x \in M_N$,
\item $l_x \cap l_y = \emptyset$ for all $x\neq y \in M_N$.   
\end{itemize}
\end{theorem}

\begin{example}\label{nikodym}
\noindent One can use the above construction to construct an optimal martingale transport, whose equivalence classes are singletons, hence the disintegration of the first marginal along the partitions $C(x)$ is the Dirac mass $\delta_x$,
 which is obviously not absolutely continuous w.r.t. $\mathcal{L}^1$.\\
 Consider the obvious inequality
$\frac{1}{2\epsilon} (|x-y| - \epsilon)^2 \geq 0$, 
and its equivalent form
\begin{align}\label{u}
\frac{1}{2\epsilon} |y|^2 \geq |x-y| + \frac{1}{\epsilon} x \cdot (y-x) + \frac{1}{2\epsilon}|x|^2 - \epsilon.
\end{align}
Thus by letting $\alpha_\epsilon(x) = \frac{1}{2\epsilon}|x|^2 - \epsilon$, $\beta_\epsilon(y) = \frac{1}{2\epsilon} |y|^2 $ and $\gamma_\epsilon(x) = \frac{1}{\epsilon} x$, \eqref{u} yields that the set $\Gamma=\{(x,y); |x-y|=\epsilon\}$ is a $c$-contact layer, where $c(x,y)=|x-y|$ in the maximization problem. It follows that every martingale $\pi_\epsilon:= (X, Y)$ with $|X-Y|=\epsilon \,\,a.s.$ is optimal with its own marginals $X \sim \mu$ and $Y \sim \nu$.\\
Now fix $\epsilon >0$ small and let $X$ be a random variable whose distribution $\mu$ has  uniform density on $[-1,1]^3$.  We define $Y$ conditionally on $X$ by evenly distributing the mass along the lines $l_x$ considered in Theorem \ref{akp} and distance $\epsilon$, that is $Y$ splits equally in two pieces from $x\in X$ along $l_x$ with distance $\epsilon$. Then the martingale $(X,Y)$ is optimal for the maximization problem. But note that in this case, each equivalence class $[x]$ is the singleton $\{x\}$, so the disintegration of $\mu$ along the partitions $C(x)$ is the Dirac mass $\delta_x$, which is obviously not absolutely continuous w.r.t. $\mathcal{L}^1$. 
Hence, the decomposition is not useful in this case.   
One also notices that the convex sets associated to the irreducible paving of the martingale $(X,Y)$ have codimension 2. We leave it as an open problem whether one can do without assumption \eqref{co1} in Theorem \ref{th: Choquet codim 1}.
\end{example}

\begin{remark} By letting $\epsilon \to 0$, the above problem approaches the one considered in  Example \ref{ah}, that is the case when the marginals $\mu = \nu$ are equal, the only maximal martingale transport is the identity, and the value of the maximal cost is zero. On the other hand, note that  
$\int \beta_\epsilon(y) d\nu(y) - \int \alpha_\epsilon(x)  d\mu(x) = \epsilon$, which means that $( \alpha_\epsilon,  \beta_\epsilon,  \gamma_\epsilon)$ is a minimizing sequence for the dual problem. But neither of the sequences  
$\alpha_\epsilon,  \beta_\epsilon,  \gamma_\epsilon$ converge (neither pointwise nor in $L^1$). This is another manifestation of the non-existence of a dual in example \ref{ah}.
This said, for the minimization problem, we have no example  where duality is not attained. 
\end{remark}

\appendix
\section{A suitable concentration set for a martingale transport plan}
Here we prove the following lemma which was introduced in Section \ref{ss: structure from dual}.

\begin{lemma}
Let $\pi \in$ MT$(\mu, \nu)$ and let $\Lambda \subset \R^d\times \R^d$ be a Borel set with $\pi(\Lambda) =1$. Then there exists a Borel set $\Gamma \subset \Lambda$ with $\pi(\Gamma) =1$ such that the map $x\mapsto \pi_x$ is measurable and defined everywhere on $X_\Gamma$ in such a way that:
\begin{enumerate}
\item $\overline{\Gamma_x} = \supp \pi_x$  for all $x \in X_\Gamma$,
\item $\Gamma \in {\mathcal S}_{MT}$, that is $x \in IC(\Gamma_x)$ for all $x \in X_\Gamma$,
\item If we assume that $\mu << \mathcal{L}^d$, then $\Gamma$ can be chosen in such a way that $X_\Gamma \subset IC(Y_\Gamma)$.
\item If in addition $\pi$ is a solution of the optimization problem \eqref{MGTP}, then $\Gamma$ can be chosen to be finitely $c$-exposable.
\end{enumerate}
\end{lemma}
\begin{proof}  Let $(\pi_x)_x$ be the unique disintegration  of $\pi$ with respect to $\mu$. It is well known that this yields a well-defined measurable map $x\mapsto \pi_x$ on a Borel set $E$ in $\R^d$ with $\mu (E)=1$ such that each $x$ in $E$ is the barycenter of $\pi_x$ and $\pi_x(\Lambda_x) =1 $.  It is clear that $x \in CC(\Lambda_x)$. However, it is not necessarily in $IC(\Lambda_x)$. Note however that 
for any Borel set $B$ in $\R^d$, the map $x \mapsto \pi_x(B)$ is Borel measurable, hence for each $r >0$, the  set $B_r:=\{(x,y) \,|\,  x \in E, \, \pi_x(B_r(y))>0\}$ is Borel (Here, $B_r(y)$ is the open ball with center $y$ and radius $r$ in $\R^d$) and consequently the set 
$\Theta := \{ (x,y) \, | \, x \in E, \, y \in {\rm supp}\, (\pi_x) \} = \bigcap_{n=1}^\infty  B_{1/n}$
  is also Borel. Letting $\Gamma := \Lambda \cap \Theta$, it is clear that  $\pi(\Gamma) = 1$ and $\pi_x(\Gamma_x)=1$ for all $x \in E$. Finally note that the probability measure $\pi_x$ has its barycenter at $x$ and that ${\Gamma_x} \subset {\rm supp}\,(\pi_x)$, and since $\pi_x(\Gamma_x) =1$, we have that $\overline{\Gamma_x} = \supp \pi_x$. Hence in particular, $x \in IC(\Gamma_x)$ for $x \in E$, proving (1) and (2).\\
Item (3) can be obtained by considering another subset of $\Gamma$. Indeed, let $X'$ be the set of Lebesgue points of $X_\Gamma$. Then as $\mu << \mathcal{L}^d$, we have  $\mu (X') = 1$. 
 Let $\Gamma' := \Gamma \cap (X' \times \R^d)$. Then,  $\Gamma' \in S_{MT}$,  $\pi (\Gamma') =1$ 
 and $X' \subset IC(X') \subset IC(Y_{\Gamma'})$, as claimed. \\
 For item (4), we use  \cite{bj}, \cite{Z}, where it is shown that for an optimizer $\pi$, there exists $\Lambda$ with  $\pi(\Lambda)=1$, that is finitely $c$-exposable (also called finitely $c$-monotone in \cite{bj}; see Definition \ref{exposable}) . We then restrict $\Lambda$ to get $\Gamma$ which also satisfies (1), (2) and (3) by the above procedure.
 \end{proof}

\section{An estimate for convex functions} 
We prove here a technical result --used in Section \ref{S: local dual}-- that allows us to control the maximum of a convex function by the integral of its second derivatives. 
Namely, 
\begin{proposition}\label{prop: Lap bound}
 Let $B_r$ denote the closed ball of radius $r$ centred at the origin $0$. Let $\phi$ be a (smooth) convex function such that $\phi (0)=0$ and $\phi \ge 0$. 
 Then, 
 \begin{align}\label{eq: Lap bound}
 \int_{B_{\sqrt{2}r}} \Delta \phi  \ge C_0 r^{d-2} \max_{B_r} \phi, 
\end{align}
where the constant $C_0>0$ depends only on the dimension $d$. 
\end{proposition}
\begin{proof}
Denote  $M_r=  \max_{B_r} \phi$.
By the maximum principle a point, say $p \in \partial B_r$,  can be chosen from the boundary so that 
$ \phi (p ) = M_r $.
 Choose an orthonormal basis $\eta_1, ..., \eta_d$ such that 
$ p = r\eta_1, $ and define a cylindrical set (of radius $r/2$)
\begin{align*}
 K_r : = \left\{  \sum_{j=1}^d t_j \eta_j  \ \Bigg| \  - r \le t_1 \le  r ,  \sqrt{\sum_{j\ne 1,} t_j^2}  \le r /2 \right\}.  
\end{align*}
We will show that 
\begin{align}\label{eq: K r est}
 \int_{K_r} D_{11}^2 \phi  \ge C_0 r^{n-2} \max_{B_r} \phi
\end{align}
for a constant $C_0>0$ depending only on the dimension $d$. This will immediately imply the desired estimate \eqref{eq: Lap bound} because $K_r \subset B_{\sqrt{2} r}$ and that $0 \le D_{11}^2 \phi \le \Delta \phi $ for the convex function $\phi$. \\
To show \eqref{eq: K r est}, we let $H$ denote the hyperplane  $\{ z_1 =0\}$. 
Notice that $\phi(0)=0$ and $\phi \le M_r $ on $B_r$, thus from convexity of $\phi$, we see that 
\begin{align}\label{eq: phi less M r}
 \phi(z) \le  \frac{|z|}{r} \left(M_r - \phi(0)\right) =  \frac{|z|}{r}M_r \quad \hbox{for each $z \in B_r$.}
\end{align}
Also, notice the fact that $p =r\eta_1$ is a maximum point of $\phi$ in $B_r$ and that  the hyperplane $r \eta_1+ H$ stays outside the interior of $B_r$, i.e. $r\eta_1 +H \subset \R^d \setminus (int B_r)$. So, from convexity of $\phi$, we have 
\begin{align*}
 \phi(z + r\eta_1) \ge M_r \quad \hbox{for each $z \in  H$.}
\end{align*}
Combining this with \eqref{eq: phi less M r} and using convexity of $\phi$ again, we can estimate the derivative $D_1 \phi$ on the set $r\eta_1 + (H\cap B_r)$. Namely, for each $z \in H\cap B_r$, 
\begin{align*}
  & D_1 \phi(z+ r\eta_1) \ge \frac{1}{r} \left( \phi(z+ r\eta_1) - \phi(z) \right) \ge  \frac{1}{r} \left( M_r - \frac{|z|}{r} M_r \right) = \frac{1}{r^2}( r- |z|)M_r. 
  \end{align*}
Similarly, use  \eqref{eq: phi less M r}, the fact that $\phi \ge 0$, in particular on $-r\eta_1+ H$, and the convexity of $\phi$  to see that 
\begin{align*}
 D_1 \phi (z - r\eta_1) \le \frac{1}{r} \left(\phi(z) - \phi (z-r\eta_1)\right) \le \frac{|z|}{r^2} M_r , \quad \hbox{ for each $z \in H\cap B_r$}. 
\end{align*}
From these estimates on $D_1 \phi$, we have that for each $z \in H\cap B_r$, 
 \begin{eqnarray*}
 \int_{-r}^{r}  D_{11}^2 \phi  ( z + t \eta_1 ) dt &=& D_1 \phi  ( z + r  \eta_1 ) - D_1\phi  ( z- r \eta_1)
 \\
 & \ge & \frac{1}{r^2}( r- |z|)M_r  - \frac{|z|}{r^2} M_r  \\
 &=& \frac{1}{r^2} (r - 2|z| ) M_r.
\end{eqnarray*}
Now, 
\begin{align*}
 \int_{K_r} D^2_{11} \phi dz 
& = \int_{z\in H\cap B_{r/2} } \int_{-r}^{r} D^2_{11} \phi (z+ t \eta_1) dt dz  \\
 & \ge  \int_{z\in H\cap B_{r/2}}  \frac{1}{r^2} (r - 2|z| ) M_r dz\\
& = C_0 r^{d-2} M_r 
\end{align*}
where 
$$
 C_0 =r^{2-d} \int_{H\cap B_{r/2}}  \frac{1}{r^2} (r - 2|z| )dz
  =   \int_{H\cap B_{1/2}}  (1 - 2|z| )dz$$
is independent of $r$.  Notice that $C_0 > 0$  because $|z|$ varies from $0$ to $1/2$ on $H\cap B_{1/2}$. 
This completes the proof.
\end{proof}

\section{A bi-Lipschitz flattening map}  

The following lemma, which describes a bi-Lipschitz ``flattening map"  was used in Section \ref{s: structure}.

\begin{lemma}
\label{4.4} Let $\R^d = V \times W$, where $V=\R^{d-1}$ and $W=\R$. Let $\delta > 0$ and let $A$ be a subset of $W$. Suppose that for each $h \in A$, there is a set $D_h$ which is contained in a hyperplane $H_h$ with $H_h \cap W = \{0,...,0,h\}$. Suppose further that $\{D_h\}_{h\in A}$ are mutually disjoint and the projection of every $\{D_h\}$ on $V$ contains the ball $B_{R}$ with center $0$ and radius $R$ in $V$. Finally, suppose that the angle between $H_h$ and $W$ is bounded; there is $\eta < \pi /2$ such that the normal direction of $H_h$ and the direction of $W$ has angle less than $\eta$ for every $h \in A$. 

Now define the flattening map $F : \cup_h D_h \to \R^d$ as follows: for $x=(v,w) \in D_h$, $F(v,w) = (v,h)$. Then $F$ is bi-Lipschitz on the set $N := (\cup_h D_h) \cap (B_r \times W)$, where $r < R$.
\end{lemma}
\begin{proof} First, note that by the disjointness of $\{D_h\}$ the map $F$ is bijective, so $F^{-1}$ is well-defined. 
The lemma is intuitively clear; the map $F$ cannot move two nearby points too far away, because the hyperdiscs $\{D_h\}$ are disjoint.\\
First of all, from the bounded angle assumption, $F$ is clearly bi-Lipschitz on each $F(D_h)$ with the same Lipschitz constant for all $h \in A$. Hence, 
for $x_1=(v_1,w_1)$, $x_2=(v_2,w_2)$, we will assume that $x_1, x_2$ are contained in $D_{h_1}, D_{h_2}$ respectively, and $h_1 \neq h_2$.\\
We consider the case $v_1 = v_2 \in V$ and $|v_1| = |v_2|  \leq r$. Let $L$ be the 1-dimensional subspace of $V$ containing $0$ and $v_1$. Regarding $D_{h_1}, D_{h_2}$ as affine functions on $V$, since their graphs on $L \cap B_R$ are disjoint and linear and $r<R$, it is clear that $|w_1 - w_2| \approx |h_1 - h_2|$; i.e.
\begin{align}\label{ba}
  C_1 |h_1 - h_2|  \leq |w_1 - w_2|\leq C_2 |h_1 - h_2|\, \text{ for some } \,C_1, C_2 >0.
\end{align}
Next we consider the case $v_1 \neq v_2$. We want to show $|x_1 - x_2| \approx |F(x_1) - F(x_2)|$, or equivalently,
\begin{align*}
|w_1 - w_2| \approx |h_1 - h_2|.
\end{align*}
Let $L$ be the 1-dimensional subspace of $V$ containing $v_1$ and $v_2$. Regarding $D_{h_1}, D_{h_2}$ as affine functions on $V$, since their graphs on $L \cap B_R$ are disjoint and linear, it is clear that 
\begin{align*}
|w_1 - w_2| = | D_{h_1}(v_1) -  D_{h_2}(v_2)|      \leq max \big{(} |D_{h_1}(v_1) - D_{h_2}(v_1)|, |D_{h_1}(v_2) - D_{h_2}(v_2)|\big{)}.
\end{align*}
But by \eqref{ba}, we have 
\begin{align*}
\max \big{(} |D_{h_1}(v_1) - D_{h_2}(v_1)|, |D_{h_1}(v_2) - D_{h_2}(v_2)|\big{)} \leq C_2|h_1 - h_2|, 
\end{align*}
which shows that $F^{-1}$ is Lipschitz on $F(N)$. On the other hand, by \eqref{ba}, we have
\begin{align*}
|h_1 - h_2|   \leq( 1 / {C_1} ) \min \big{(} |D_{h_1}(v_1) - D_{h_2}(v_1)|, |D_{h_1}(v_2) - D_{h_2}(v_2)|\big{)}
\end{align*}
and, again regarding $D_{h_1}, D_{h_2}$ as disjoint linear graphs on $L \cap B_R$, we have
\begin{align*}
\min \big{(} |D_{h_1}(v_1) - D_{h_2}(v_1)|, |D_{h_1}(v_2) - D_{h_2}(v_2)|\big{)} \leq | D_{h_1}(v_1) -  D_{h_2}(v_2)| = |w_1 - w_2|  
\end{align*}
which shows that $F$ is Lipschitz on $N$, and the proof is complete.
\end{proof}

\section{Proofs of measurability}
%

We now establish  the following proposition which was used in the proofs of Section \ref{s: structure}. 

\begin{proposition}\label{prop:measurable E}
Let $\pi$ be a Borel measure on the product space $\R^d \times \R^d$  and let $A \subset \R^d$ be a concentration set for its first marginal. Let $x \mapsto \pi_x$ be the corresponding disintegration map from  $A$ to $P(\R^d)$ and assume that for each $x \in A$, the set  ${\rm supp\,} \pi_x \subset \R^d $  is compact -- which is satisfied in particular, if the second marginal of $\pi$ is compactly supported.
 Then, the set
\begin{align*}
 E_\pi :=\{ x \in  A \ |  \  {\rm supp \, } \pi_x \subset {\rm Ext} ( \overline{\rm conv} ({\rm supp \, } \pi_x ))\}
\end{align*}
 is a Borel measurable set in $\R^d$.  
\end{proposition}

\begin{proof}
Let $N_\pi =  A \setminus E_\pi$, that is, 
\begin{align*}
 N_\pi =\{ x \in  A \ |  \  {\rm supp \, } \pi_x \not \subset {\rm Ext} ( \overline{{\rm conv}} ({\rm supp \, } \pi_x ))\} . 
\end{align*}
We will show that there is a measurable set $N$ in $\R^d$ such that $N_\pi \subset N$ and $E_\pi \cap N =\emptyset$, which then implies that the set  $E_\pi = A  \setminus N$ is measurable, as desired.

We shall use a classical result of Carath\'eodory, which implies that 
a point $z\in {\rm supp}\, \pi_x$ is not an extremal point of the convex hull of ${\rm supp \, } \pi_x $ if and only if it 
 lies in the relative {\em interior} of an $r$-simplex ($1\le r\le d$) with vertices in ${\rm supp \, } \pi_x$. 
First choose a countable dense subset $Q \subset \R^d$ and associate to each $q \in Q$ 
an $(\epsilon, \delta)$-admissible $r$-simplex $S \subset \R^d$, defined as follows: 
\begin{enumerate}
\item all the vertices of $S$ belong to $Q$, 
 \item $q$ is $\epsilon$-close to a (relative) interior point of $S$, and
 \item  all vertices of $S$ are $\delta$-away from $q$. 
\end{enumerate}
Let ${\mathcal A}_{\epsilon, \delta}(q)$ denote the countable set of  all $(\epsilon, \delta)$-admissible simplices for $q$. 
Now define the set 
\begin{align*}
 S_{\epsilon, \delta} (q) & : = \{  x \in A  \ |   \  \pi_x (B_\epsilon (q)) >0  \  \& \  \hbox{there exists $S\in {\mathcal A}_{\epsilon, \delta}(q)$ such that}\\
 & \quad \quad \quad \quad \quad \hbox{ for each vertice $q_j$ of $S$, }  \pi_x  (B_\epsilon (q_j)) >0  \}. 
\end{align*}
This set $S_{\epsilon, \delta} (q)$ contains all those points $x$ in $A$, such that ${\rm supp}\, \pi_x$ include, up to an $\epsilon$-error, both the point $q$ and the vertices of an $(\epsilon, \delta)$-admissible simplex for $q$. 
 Since the map $x \mapsto \pi_x \in P(\R^d)$ is measurable, each set $S_{\epsilon, \delta} (q)$ is measurable, since it can be written as the countable union of measurable sets. 
Define 
$ N_{\epsilon, \delta} : = \bigcup_{q\in Q} S_{\epsilon, \delta} (q), $
 and set
 \begin{align*}
N= \bigcup_{k \ge 1} \bigcap_{j \ge 1} N_{2^{-j-k}, 2^{-k}}. 
\end{align*}
It is clear that $N$ is measurable. We now show that it has the desired properties. 

\vspace{3mm}
\noindent{\bf Claim 1:} $N_\pi \subset N$. 
Indeed, for any $x \in N_\pi$,  there exists a $z \in {\rm supp\, } \pi_x$ lying in the relative interior of an $r$-simplex, say $S$,  with vertices in  ${\rm supp \, } \pi_x$. 
Let $\delta_0 >0$ be a lower bound for the distances from $z$ to the vertices of $S$ as well as the distances between any two vertices. 
Fix $k \in \N$ large enough so that $\delta_0 \ge \delta = 2^{-k+1}$. 
Since $Q$ is dense, one can find  for each $\epsilon=2^{-j -k}$, $j \ge 1$, a point $q \in B_\epsilon (z)$ and 
an $(\epsilon, \delta)$-admissible simplex $S_j$ for $q$ whose vertices are $\epsilon$ close to the vertices of $S$. This implies that for each $j \in \N$, $x \in S_{\epsilon, \delta} (q)$ where $\epsilon=2^{-j -k}$ and $\delta = 2^{-k+1}$ . This shows that $x \in \bigcap_{j\ge 1} N_{2^{-j-k}, 2^{-k}} \subset N$ as desired. 

\vspace{3mm}

\noindent {\bf Claim 2:} $E_\pi \cap N =\emptyset$. 
Indeed,  suppose not  
then there exists $x \in E_{\pi} \cap  \bigcap_{j \ge 1} N_{2^{-j-k}, 2^{-k}}$ for some $k \in \N$. Let $\delta = 2^{-k}$ and $ \epsilon_j =2^{-j-k}$ for each $j\ge 1$. Then, we see that 
 for each $j \ge 1$, there exists  $q_{j} \in Q$ and a simplex, say $S_j$, that is $(\epsilon_j, \delta)$-admissible for $q_j$ such that $q_j$ and the vertices of $S_j$ are $\epsilon_j$ close to ${\rm supp}\, \pi_x$. Since ${\rm supp}\, \pi_x$ is compact by assumption, there exists a convergent subsequence of $\{q_j\}$,  as well as a convergent subsequence of the simplices $\{S_j\}$ (in the Hausdorff topology since their vertices converge). Let $q_\infty$, $S_\infty$ denote their limits (as $j \to \infty$), respectively. 
Note  that $q_\infty \in {\rm supp}\, \pi_x$ and that  $S_\infty$ is a simplex with vertices in ${\rm supp}\, \pi_x$.  
By the definition of $(\epsilon_j, \delta)$-admissibility, we also have that 
  $q_\infty$ belongs to the closure of $S_\infty$, while being $\delta$-away from its vertices. This implies that 
  the point $q_\infty \in {\rm supp}\, \pi_x$ is not an extremal point of the convex hull of ${\rm supp}\, \pi_x$. This contradicts  the fact that $x\in E_\pi$, thus completing the proof of Claim 2 and the proposition.  
\end{proof}

Next, we show the following.

\begin{proposition}\label{prop:meas w}
Let $\pi$ be a Borel measure on the product space $\R^d \times \R^d$  and let $x   \mapsto \pi_x \in P(\R^d)$ be its disintegration along the first marginal. Let $A\subset \R^d$ be a Borel measurable set  that is a concentration set for the first marginal of $\pi$, and denote for each $x \in A$, the set $I(x) = IC ({\rm supp}\, \pi_x)$.

 Assume that $I(x)$ is bounded and that $x \in I(x)$ for each $x \in A$. If $A_1 \subset A$ is a measurable set such that for each $x \in A_1$,  $\dim I(x) =1$, then   
 the function $w: A_1 \to \R_+$ defined by 
 \begin{align*}
 w (x) : = \min [ \dist(x, y_0), \dist(x, y_1)]
\end{align*}
where $y_0, y_1$ are the end points of the segment $I(x)$, is Borel measurable.
 \end{proposition}

\begin{proof}
It is enough to show that for each $\delta>0$, the set 
$ M_\delta = \{ x \in A_1 \ | \ w (x) \ge \delta \} $
is Borel measurable. 
For that, we again consider
a countable dense subset $Q \subset \R^d$. For $\epsilon, \delta >0$, we say that a closed segment $S=[q_0, q_1]$ connecting two points $p_0, p_1 \in \R^d$ is  $(\epsilon, \delta)$-admissible  for $q \in Q$
if 
\begin{enumerate}
\item $p_0, p_1 \in  Q$;
 \item $q \in N_\epsilon (S)$, the latter being the $\epsilon$-tubular neighborhood of $S$; 
 \item  $ \dist (p_i, q) \ge \delta  $, for $i=0,1$. 
 \end{enumerate}
Let ${\mathcal A}_{\epsilon, \delta}(q)$ denote the countable set of  $(\epsilon, \delta)$-admissible segments for $q$, 
and define the set 
\begin{align*}
 S_{\epsilon, \delta} (q) & : = \{  x \in A_1  \ |  \  \dist (x, q) \le \epsilon \  \& \  \hbox{there exists $[p_0, p_1] \in {\mathcal A}_{\epsilon, \delta}(q)$ with }  \pi_x  (B_\epsilon (p_i)) >0 , \ \ i=0,1 \}. 
\end{align*}
The set $S_{\epsilon, \delta} (q)$ contains those points $x$ in $A_1$, such that $x$ is $\epsilon$-close to $q$, while 
  ${\rm supp}\, \pi_x$ includes up to $\epsilon$,  
the end points of an $(\epsilon, \delta)$-admissible segment for $q$. 
Again, each set $S_{\epsilon, \delta} (q)$ is measurable, since the map $x \mapsto \pi_x \in P(\R^d)$ is measurable.
Define the  set 
$ M_{\epsilon, \delta} : = \bigcup_{q\in Q} S_{\epsilon, \delta} (q)$, 
and set
 \begin{align*}
\bar M_\delta= \bigcap_{j \ge 1} M_{2^{-j}, \delta}. 
\end{align*}
It is obvious that $\bar M_\delta$ is measurable.
We claim that 
\begin{align}\label{M}
 M_\delta = \bar M_\delta  . 
\end{align}
Indeed, we first verify that $\bar M_\delta \subset M_\delta$. To see this, consider an arbitrary point $x \in \bar M_\delta$, and let $y_0, y_1$ be the two end points of the segment $I(x)$.  Then for each $0< \epsilon < \delta$, there is $q\in Q$ and $S= [p_0, p_1] \in {\mathcal A}_{\epsilon/3, \delta} (q)$ such that $x \in B_{\epsilon/2} (q)$ and $\pi_x (B_{\epsilon/2}(p_i)) >0$ for $i=0,1$. 
From the last condition, we see that $p_0, p_1  \in N_{\epsilon/2} ({\rm supp}\, \pi_x)$ and hence $S \in N_{\epsilon/2} (I(x))$. Moreover, from the item (3) for the $(\epsilon, \delta)$-admissibility of $S$ together with $x \in B_{\epsilon/2} (q)$, we see that 
 $\dist(p_i, x) \ge \delta-\epsilon/2$, 
which
 then implies that
$\dist(x, y_i)  \ge \delta -\epsilon.$
Since $\epsilon >0$ was arbitrary, this  implies that $x\in M_\delta$ as desired. \\
For the reverse inclusion $M_\delta \subset  \bar M_\delta$, note that  for each $x \in M_\delta$, we have $\dist(y_i, x) \ge \delta $, $i=0,1$, where $y_0, y_1$ are the end points of the segment $I(x)$. 
Also, notice that $y_i \in {\rm supp}\, \pi_x $, $i=0,1$. 
Since $Q \subset \R^d$ is dense, one can find  for each $0< \epsilon < \delta $,  a point $q\in Q$ and a segment $S=[p_0, p_1]  \in A_{\epsilon, \delta} (q)$ such that $q \in B_\epsilon (x)$, and $p_i \in B_\epsilon(y_i))$, $i=0,1$. It follows that 
$x \in S_{\epsilon, \delta} (q)$
which implies
$ x \in M_{\epsilon, \delta} $ for all $0< \epsilon < \delta, \hbox{ thus } x \in \bar M_\delta$.
This completes the proof.
\end{proof}

\end{document}